\documentclass{aptpub}

\newtheorem{hyp}[thm]{Hypothesis}
\newtheorem{pro}[thm]{Proposition}

\newcommand{\HH}{\mathbb{H}}
\newcommand{\FF}{\mathbb{F}}
\newcommand{\F}{\mathcal{F}}

\newcommand{\DD}{\mathbb D}
\newcommand{\LL}{\mathcal{L}}
\newcommand{\II}{\mathcal{I}}

\newcommand{\h}{{\cal H}}

\newcommand{\demi}{\frac{1}{2}}

\newcommand{\1}{{\mathbf 1}}

\newcommand{\biindice}[3]

\shorttitle{PDE for  the joint law of the pair of a continuous diffusion and its running maximum}

\begin{document}

\title{PDE for  the joint law of the pair of a continuous diffusion and its running maximum}

\authorone[IMT]{Laure Coutin, Monique Pontier}

\addressone{IMT, Universit\'e Paul Sabatier, 31062 Toulouse cedex France}  

\footnote {coutin@math.univ-toulouse.fr, IMT.} 

\footnote {pontier@math.univ-toulouse.fr, IMT.}


\begin{abstract}
Let $X$ be a $d$-dimensional diffusion   and $M$ the  running supremum of its first component.
 In this paper,  { we show} that for any $t>0,$
  the  density  (with respect to  the $d+1$-dimensional Lebesgue measure) of  the pair 
  $(M_t,X_t)$
is a weak solution of a Fokker-Planck partial differential equation on the closed set $\{(m,x)\in \mathbb{R}^{d+1},~{ m\geq x^1}\},$ {using an integral expansion of this density.}
\end{abstract}

\begin{keywords}{ Diffusion process, partial differential equation, running supremum process, joint law.}
\end{keywords}

\ams{60J60}{60G70,60H10}

\section{Introduction}

{The goal of this  paper   is to study the law of the   pair  $(M_t, X_t)$ where $X$ is a $d$-dimensional diffusion
and $M$ is the running maximum of the first component.
In a previous work \cite{LaureMo}, using Malliavin calculus 
  and specifically Nualart's seminal book \cite{nualart}, we  have proved that, for any $t> 0$ the law of  {$V_t:=(M_t,X_t)$}  is absolutely continuous with respect to the  Lebesgue
measure with density $p_V(.;t)$,  and that  the support of this density   is included in
 the  set  ${\{(m,x)\in \mathbb{R}^{d+1},~m\geq x^1\}}$.

 In the present work, we prove that
the density $p_V$ is a weak solution { of} a
 partial differential equation (PDE).
 Furthermore,  we exhibit a boundary  condition on the set $\{(m,x)\in \mathbb{R}^{d+1},~m=x^1\}.$
 This work  extends 
 the results given in    \cite{WalyLaMo} and in  Ngom's thesis \cite{waly} obtained in the case where $X$ is  a L\'evy process  and  where it is proved that the density is a weak solution to an integro-differential equation.
\\
 In the literature, there exist
many  studies on the law of  $V_t$.
 When the process $X$ is a Brownian motion, one can refer to \cite{HKR, CJY}  where an explicit expression of $p_V$ is given.
 When $X$ is a one-dimensional linear diffusion,
\cite{CsFolSal} provides an expression of $p_V$
using the scale function, the speed measure
and the density of the law of some hitting times. See also \cite{patie,dorobantu} for the particular case of Ornstein-Uhlenbeck process.
   For some applications to the    { local score of a biologic sequence}, the case of reflected Brownian motion  is  presented in \cite{LagMerVal}. The law of the maximum  $M_t$ is studied in \cite{azais}
 for general Gaussian processes. The case of a L\'evy process $X$ is
  deeply investigated in the  literature,
see for instance \cite{DonKyp,waly}. Moreover
 Section 2.4 in  Ngom's thesis \cite{waly} provides
 the existence and the regularity of the joint law density of the process $(M_t,X_t)$ for a L\'evy process $X.$
 {In the case where $X$ is a martingale (see e.g. \cite{Rog,DuemRog} or  \cite{cox,HLOST}), the law of the running maximum is provided.}
}
 Such studies concerning this running maximum
 are useful in financial area which involve hitting times,  for instance for the pricing of barrier option.
 It is known that the law of  hitting times is closely related to the one
 of the running maximum, see \cite{BHR,
  coutin-dorobantu-2011,volpi}.
 { As an application of our work,
 think of a firm the activity of which is characterized by a set of                               processes $(X^1,\cdots,X^d)$. But one of them, e.g., $X^1$ could be linked to an alarm, namely: when there exists $s\leq t$ such
that $X^1_s$  exceeds a threshold $a,$
that is equivalent to $M_t=\sup_{0\leq s\leq t} {{X_s^1}}\geq a,$ some action is important
to operate. So, the firm needs to know the law of such pair
$(M_t,X_t)$;  more specifically the law of the stopping time $\tau_a= \inf\{u, X^1_u \geq a\}$,
is linked to the law of $M$ as following:
$\{\tau_a\leq t\}=\{M_t\geq a\}.$
To know the probability of such an alert, the law of the pair $(M_t,X_t)$  will be useful.}

{ We provide an infinite expansion of the density of the law of the pair $(M_t,X_t)$ which can leads to numerical approximation.}

 Let $(\Omega,{\cal F}, {\mathbb P})$ be a probability space 
 endowed with a $d$-dimensional  Brownian.
{ Let $X$ be the diffusion process with values in  $\mathbb{R}^d$ solution of }
\begin{align}
\label{modelm}
dX_t= B(X_t) dt + A(X_t)dW_t,~~{ t>0}
\end{align}
where   $X_0$ is  a random variable independent of the Brownian motion $W$, with law $ \mu_0$,     and $A$  (resp.  $B$) is a map from $\mathbb{R}^d$ to   the set of  $(d\times d)$ matrices  (resp. to $\mathbb{R}^d).$
 {Let us denote}  $C^i_b(\mathbb{R}^d,\mathbb{R}^n)$ the set of  functions { on $\mathbb{R}^d$, which are $i$ times differentiable}, bounded, with bounded derivatives, taking their values in $\mathbb{R}^n$.
{Let $\FF=(\F_t,t\geq 0)$ be the  completed right-continuous filtration defined by  $  \F_t:=\sigma\{X_0,~W_s,s\leq t\}\vee {\mathcal N}$ where ${\mathcal N}$ is the set of negligible sets of ${\mathcal F}.$ }

{ Under classical assumptions on $A$ and $B$
(cf.\eqref{h1h2} and \eqref{hyp:UnEl} below), then according to \cite{LaureMo},
for all $t>0$, the law of $V_t=(\sup_{u\leq t } X^1_t,X_t)$ has a density with respect to the Lebesgue measure on ${\mathbb R}^{d+1}.$}

  The main results and notations are given in Section \ref{sec1}:
 in  the $d$-dimensional case, under  a quite natural assumption (meaning Hypothesis \ref{hyp-dens-outil} below) on the regularity of $p_V$ around the boundary of $\Delta,$  $p_V$ is a weak  solution of a Fokker-Planck PDE on the  subset
 of $\mathbb{R}^{d+1}$ defined by $\{(m,x),~m\geq x^1\}.$ When $A=I_d,$ this assumption is satisfied, see Theorem \ref{existence-dens}.
 The main results are proved  in Section \ref{ApA} under  Hypothesis \ref{hyp-dens-outil}. Section \ref{density} is devoted to prove  that Hypothesis \ref{hyp-dens-outil} is satisfied when $A= I_d.$
 { The main tool is  an infinite expansion of $p_V$ given  in Proposition \ref{pro-forme-dens-phi-t}.}
 In Section  \ref{sec5},
one-dimensional case, a Lamperti transformation \cite{lamp}  allows to get the main result for any $A\in C^2_b(\mathbb{R},\mathbb{R}).$  
Finally   Appendix
 contains   {
some technical tools useful for the  proofs of main results.}

\section{Main results { and some notations}}
\label{sec1}

{{In this section, we give our main results, the proofs will be given later on, as it is mentioned in the introduction.

\subsection{Notations}

 { Let $\Delta$ be the open set of $\mathbb{R}^{ D+1}$ given by } $\Delta:=  \{ (m,x), m\in \mathbb{R} , x \in\mathbb{R}^{d}, m> x^1, x=(x^1,\cdots,x^d)\}.
$
{ From now on, we use Einstein's convention.}
 The infinitesimal generator ${\mathcal L}$ of the diffusion $X$ defined in  \eqref{modelm} 
 { is the partial differential operator }
 on { the space $C^2(\mathbb{R}^d,\mathbb{R})$}  given by:
\begin{align}
\label{def-genm}
{\cal L}= B^i\partial_{x_i}+\demi (A A^t)^{ij}\partial^2_{x_i,x_j}.
\end{align}}
{where   $A^t$ denotes the transposed matrix.}
\\
 Its adjoint operator is
 ${\LL^*f=
 \demi \Sigma^{ij}\partial^2_{ij}f
 -[B^i-\partial_j(\Sigma^{ij})]\partial_if
 -[\partial_iB^i-\demi\partial^2_{ij}
 (\Sigma^{ij})]f}$ where $\Sigma:=AA^t.$
 In what follows, 
the operators $\LL$  and $\LL^* $ are extended to the space  $C^2(\mathbb{R}^{d+1},\mathbb{R})$, for $\Phi \in C^2(\mathbb{R}^{d+1},\mathbb{R})$ as 
$$\LL(\Phi)(m,x)=B^i(x)\partial_{x_i}\Phi(m,x)+\demi \Sigma^{ij}(x)\partial^2_{x_i,x_j}\Phi(m,x),$$
and $\LL^*(\Phi)(m,x)=$
$${\demi\Sigma^{ij}(x)\partial^2_{ij} \Phi(m,x)-[B^i-\partial_j(\Sigma^{ij})](x)\partial_{x_i}\Phi(m,x)+[\demi \partial^2_{x_i,x_j}\Sigma^{ij}-\partial_{x_i}B^i](x)\Phi(m,x)}.$$
It can be stressed that these operators
are degenerated since no derivative with respect to the variable $m$ appears.}
\\
 Let $A^1(x)$ be the $d$ dimensional vector $A^1(x)=(A^1_j(x),j=1,...,d)\in \mathbb{R}^d$ corresponding to the first column of $A(x),$
{similarly  $A_j(x)$ denotes its $j$th line.} 
\\
{ Recall that
  $M$ denotes} the running maximum of the first component of  $X$, meaning
$M_t=\sup_{0\leq s\leq t} \{X^1_s\}$
and  $V$ is the  {$\mathbb{R}^{d+1}$-valued process defined  by $(V_t=(M_t, X_t), \forall t\geq 0).$} { Finally, $\tilde x\in\mathbb{R}^{d-1}$  denotes the vector $(x_2,...,x_d).$
}

 In \cite{LaureMo}, {under Assumptions \eqref{h1h2} and \eqref{hyp:UnEl}} below, when the initial value is deterministic,  $X_0=x_0 \in {\mathbb R}^d$,
the density of $V_t$  exists and is denoted $p_V(.;t,x_0)$.  If $\mu_0$ is the distribution of $X_0$, the density of the law of $V_t$  with respect to the Lebesgue measure on ${\mathbb R}^{d+1} $  is 
\begin{align}
\label{(*)}
p_V(.;t,\mu_0):=\int_{{\mathbb R}^d} p_V(.;t,x_0)d\mu_0(x_0)
\end{align}
 When {there is no ambiguity},  the dependency in $\mu_0$ {is omitted.}
\\
 Since $M_t \geq X^1_t,$ the support of $p_V(.;t,\mu_0)$ is contained in
$\bar{\Delta}:=\left\{(m,x)\in {\mathbb R}^{d+1} | m\geq x^1\right\}.$
\\
\subsection{Mains results}
 The aim of this article is to show that the density
$p_V$ is a weak solution of a Fokker-Planck PDE.
{ The coefficients $B$ and $A$ are assumed to } satisfy
\begin{equation}
\label{h1h2}
 B \in C^1_b(\mathbb{R}^d,\mathbb{R}^d)\mbox{ and } { A\in C^2_b(\mathbb{R}^d,\mathbb{R}^{d\times d}) },~
\end{equation}
 and   that there exists a constant $c>0$ such that { the Euclidean norm of any vector $v$ satisfies}
\begin{align}
\label{hyp:UnEl}
c\|v\|^2\leq  v^tA(x)A^t(x) v ,~~\forall v,x\in {\mathbb R}^d.
\end{align}

{Our first result will be established under the following hypothesis
which is a quite natural assumption on the regularity of $p_V$ { in the neighbourhood of the boundary of $\Delta$}
 since the {  set of times} where the process
$M$ {  increases} { is included in } the set} $\{t,~~M_t=X_t\}:$
\begin{hyp}
\label{hyp-dens-outil}
The density   of the {law} of $V_t=(M_t,X_t),$ {denoted   by  $p_V$
 \eqref{(*)},}
   satisfies  
    \\
(i) {the map} $(t,m,\tilde{x})\mapsto \sup_{u >0} p_V(m,m-u,\tilde{x};t)$ belongs to $L^1([0,T] \times {\mathbb R}^d,dtdmd\tilde{x}).$
   \\
(ii) for all $t>0$ 
almost surely in $(m,\tilde{x}) \in {\mathbb R}^d,$ $\lim_{u\rightarrow 0^+} p_V(m,m-u,\tilde{x};t)$ exists and is denoted by $ p_V(m,m,\tilde{x};t).$
\end{hyp}
\begin{thm}
\label{thme-sous-hyp-dens}
Assume that $A$ and $B$ fulfil  \eqref{h1h2} and \eqref{hyp:UnEl} and {that} $(M,X)$ fulfils  Hypothesis \ref{hyp-dens-outil}.
{ Then, for all initial law $\mu_0$} and $F \in C^2_b({\mathbb R}^{d+1},{\mathbb R})$:
\begin{align}
\label{PDEweakm}
{\mathbb E}\left[F(M_t,X_t)\right)& ={\mathbb E} \left[ F(X^1_0,X_0)\right] + \int_0^t {\mathbb E} \left[ {\mathcal L}\left(F\right)(M_s,X_s)\right] ds\nonumber
\\
&+\frac{1}{2} {\int_0^t
{\mathbb E}\left[
\partial_mF(X^1_s,{X_s})
\|A^{1}(X_s)\|^2
\frac{p_V(X^1_s,{X_s};s)}
{p_{X}(X_s;s)}\right] ds.}
\end{align}
\end{thm}

{ Actually  $p_X$ is the solution of the PDE
$ \partial_t p= {\mathcal L}^* p,~~p(.;0)=\mu_0,$
where 
\\
${ \LL^*f=
 \demi \Sigma^{ij}\partial^2_{ij}f
 -[B^i-\partial_j(\Sigma^{ij})]\partial_if
 -[\partial_iB^i-\demi\partial^2_{ij}
 (\Sigma^{ij})]f}$.}
Let $a_{ij}:=\Sigma^{ij},$ 
\\
$a_i:=[B^i-\partial_j(\Sigma^{ij})]\partial_i,$ and $ { a_0:=\partial_iB^i-\demi\partial^2_{ij}(\Sigma^{ij})}.$  Under Assumptions \eqref{h1h2} and \eqref{hyp:UnEl}, the operator $\LL^*$ satisfies all the assumptions of  Theorem  3.5
\cite{GarMenal} (see (3.2) (3.3) 3.4) page 177). As a consequence of Theorem 3.5  line 14  $p_X(x;s)>0$.

{ \begin{rem}
(i)  When $A$ is the identity matrix of ${\mathbb R}^d$  (denoted by $I_d$)
 and  $B\in C^1_b(\mathbb{R}^d,\mathbb{R}^d)$,   Hypothesis
\ref{hyp-dens-outil} is fulfilled, see { Theorem \ref{existence-dens} }below.
{ When $d=1$, using a Lamperti transformation 
\cite{lamp}, 
 one proves that Hypothesis
\ref{hyp-dens-outil} is always fulfilled, see Section \ref{sec5}.
}

(ii)
 This result is similar to  Theorem 2.1 
in \cite{WalyLaMo} where  the process
 $X$ is a L\'evy process.
 Proposition 4 in \cite{WalyLaMo}  
 gives a key of the last term in \eqref{PDEweakm}
 with factor $\demi$. Firstly,  roughly speaking, 
 the local behaviour of $X_t^1-X_s^1$
  conditionally to ${\mathcal F}_s$ is the
   one of  $\|A^{1}(X_s)\|(W^1_{t}-X_s^1).$
 So, as in the Brownian case, the running maximum $M$ of $X^1$
 is increasing as soon as it is equal to $X^1$ and both $M$ and $X^1$ are increasing; it is well known that the Brownian process $W^1$ is increasing with probability $\demi,$
  more specifically, we have { $\mathbb{P}\{\lim_{t\rightarrow s+}\frac{W^1_t -W_s^1}{t-s}=-\infty\}=\mathbb{P}\{\lim_{t\rightarrow s+}\frac{W^1_t -W_s^1}{t-s}=+\infty\}
 =\demi$.}
 \end{rem}

The starting point of the proof of Theorem \ref{thme-sous-hyp-dens}
is the {It\^o's} formula: 
 let $F$ belong to $C^2_b(\mathbb{R}^{d+1},\mathbb{R})$.
The process $M$ is increasing, {hence} $V=(M,X)$ is a semi-martingale.  Applying  It\^o's formula to $F(V)$ and taking expectation of both members,
\begin{align*}
{\mathbb E} \left[F(V_t)\right]={\mathbb E} \left[F(V_0)\right] +\int_0^t {\mathbb E}\left[
{\mathcal L}(F)(V_s)\right]ds
+{\mathbb E}\left[ \int_0^t \partial_mF(V_s) dM_s\right].
\end{align*}
}
{ The novelty comes from the third term of the right member of the previous equation.
The following theorem proved in Section \ref{ApA} achieves the proof of Theorem \ref{thme-sous-hyp-dens}.}
\begin{thm}
\label{pro-sous-hyp-dens}
Assume that $A$ and $B$ fulfil \eqref{h1h2} and \eqref{hyp:UnEl} and {that} $(M,X)$ fulfils Hypothesis \ref{hyp-dens-outil}.  {For all $\Psi \in C^1_b({\mathbb R}^{d+1},{\mathbb R})$, let $F_\psi$ be the map}
\\
 $F_\psi:t\mapsto {\mathbb E}
 \left[\int_0^t \Psi(M_s,X_s)dM_s\right].$
  Then $F_{\Psi}$ is absolutely
  continuous with respect to the Lebesgue measure and its derivative is
\begin{align*}
\dot{F_{\Psi}}(t)=\frac{1}{2}
\int_{{\mathbb R}^d}
{ \Psi}(m,m,\tilde{x}) \|A^{1}(m,\tilde{x})\|^2p_V(m,m,\tilde{x};t) dm d\tilde{x} .
\end{align*}
\end{thm}
Remark that, as it is expressed in Theorem \ref{thme-sous-hyp-dens}, this derivative {can} be written
$$
\frac{1}{2} 
{\mathbb E}\left[
\Psi(X^1_t,{X_t})
\|A^{1}({X_t})\|^2
\frac{p_V(X^1_t,{X_t};t)}
{p_{X}( X_t;t)}\right] 
.$$

\begin{rem}
{The above   proposition} provides an explicit formulation of the derivative of the function $F_\Psi.$ {Note that} the absolute continuity of $F_\psi$ could be established as a direct consequence of
the existence of the density of the {law} of
the hitting time $\tau_a=\inf\{s:X^1_s\geq a\}$
when it exists, using the identity
$\{\tau_a\leq t\}=\{M_t\geq a\}.$
\\
Conversely, it could be proved that {the} absolute continuity of $F_\Psi$ yields the existence of
 the density of the {law} of
the hitting time {$\tau_a$}, using a sequence of $C^2_b(\mathbb{R},\mathbb{R})$ functions
$(F_n)$ approximating the indicator function $\1_{[a,\infty)}$, namely this density satisfies
$f_{\tau_a}(t)=\demi \int_{\mathbb{R}^{d-1}}p_V(a,a,\tilde x;t)d\tilde x.$
\end{rem}

\begin{thm}
\label{existence-dens}
Assume that $A=I_d$ and $B$  satisfies Assumption
\eqref{h1h2} then, for all $t>0$ the distribution of the pair $(M_t,X_t)$ { fulfils  Hypotheses \ref{hyp-dens-outil}}. As a consequence,
{ for all $F\in C^2_b({\mathbb R}^{d+1},{\mathbb R})$
\begin{align*}
&{\mathbb E}\left[F(M_t,X_t)\right] ={\mathbb E} \left[ F(X^1_0,X_0)\right] + \int_0^t {\mathbb E} \left[ {\mathcal L}\left(F\right)(M_s,X_s)\right] ds
\\
&+\frac{1}{2} \int_0^t
{\mathbb E}\left[
\partial_mF(X^1_s,{X_s})
\frac{p_V(X^1_s,{X_s};s)}
{p_{ X}( X_s;s)}\right] ds.
\end{align*}}
\end{thm}
\begin{proof}
This theorem is a consequence of Theorem 
\ref{thme-sous-hyp-dens} and Proposition 
\ref{cont-ui--dens}.
\end{proof}

When $d=1$ a Lamperti transformation { leads to} the following corollary:
\begin{corollary}
Assume that $d=1$, $A$ and $B$ satisfies
\eqref{h1h2} and \eqref{hyp:UnEl}, the density $p_V$
satisfies Hypothesis \ref{hyp-dens-outil}
so
\begin{align*}
&{\mathbb E}\left[F(M_t,X_t)\right] ={\mathbb E} \left[ F(X_0,X_0)\right] + \int_0^t {\mathbb E} \left[ {\mathcal L}\left(F\right)(M_s,X_s)\right] ds
\\
&+\frac{1}{2} \int_0^t
{\mathbb E}\left[A^2(X_s)
\partial_mF(X_s,{X_s})
\frac{p_V(X_s,{X_s};s)}
{p_{ X}( X_s;s)}\right] ds.
\end{align*}
\end{corollary}

\begin{rem}
\label{rem-wpde} If $p_V$ is regular enough, and if the initial law of $X_0$ satisfies  $\mu_0(dx)=f_0(x)dx$, then Theorem \ref{thme-sous-hyp-dens} means that $p_V$ is a weak solution in the set $\Delta$ of 
$
\partial_t p= \LL^*p 
$
where $\LL^*f=
 \demi \Sigma^{ij}\partial^2_{ij}f
 -[B^i-\partial_j(\Sigma^{ij})]
 \partial_if
 -\partial_iB^i-\demi\partial^2_{ij}
 (\Sigma^{ij}))f$
with boundary condition
{
\begin{equation}
\label{CB1}
B^1(m,\tilde x) p_V(m,m,\tilde x;s)=
\partial_{x_k}(\Sigma^{1,k} p_V)(m,m,\tilde x;s)+\demi {\partial_{m}(\|A^1\|^2} p_V)(m,m,\tilde x;s).
\end{equation}
   This result is proved in Appendix
    \ref{ap3}

}
This boundary condition  also   appears  in   Proposition 4 Equation (11) of
 \cite{dorobantu} (Ornstein Uhlenbeck process).
 Finally,  a similar PDE is studied in  Chapter 1.2 of 
  \cite{GarMenal}
   {where the authors have}
    established the existence of a unique strong solution of this PDE,
     but { in case of} a non degenerate elliptic 
     operator. 
\end{rem}

\section{ Proof of Theorem \ref{pro-sous-hyp-dens}}

\label{ApA}
{ We start this section with a road map of the proof of Theorem \ref{pro-sous-hyp-dens}.}
Firstly we compute the right derivative of the application $F_\Psi:t\mapsto \mathbb{E}[\int_0^t\Psi(M_s,X_s)dM_s],$ namely 
  $\lim_{h\to 0^+}T_{h,t}$ with 
$T_{h,t}=\frac{1}{h}
\mathbb{E}_{\mathbb{P}}[\int_t^{t+h}\psi(V_s)dM_s].$ A first step is the decomposition
\begin{align}
\label{Tht}
T_{h,t}=\frac{1}{h}\mathbb{E}_\mathbb{P}[\int_t^{t+h}(\psi(V_s)-\psi(V_t))dM_s]+\frac{1}{h}\mathbb{E}_\mathbb{P}[\psi(V_t)(M_{t+h}-M_t)].
\end{align}
{ Since $\psi\in C^1_b(\mathbb{R}^{d+1},\mathbb{R})$
 and the process $M$
is increasing, the first term in \eqref{Tht}, }
is  dominated by:
$$
\mathbb{E}\left[\int_t^{t+h}(\psi(V_s)-\psi(V_t))dM_s\right]\leq
\|\nabla\psi\|_\infty \mathbb{E}\left[\sup_{t\leq s\leq t+h}
\|V_s-V_t\|(M_{t+h}-M_t)\right].
$$

 Lemma \ref{maj-x-unif-h} states that $\sup_{t\leq s\leq t+h}\|X_s-X_t\|_p= {O(\sqrt h)} $ and Lemma  \ref{maj-sup-h} yields 
 \\ $\|M_{t+h}-M_t\|_p={o(\sqrt h)}$  so that
 that {this first   term} is an $o(h)$.
 
Concerning the second term in \eqref{Tht},  $M_{t+h}-M_t$ can be written as $\sup_{0\leq u\leq h} (X^1_{t+h}-X^1_t-M_t+X^1_t)_+$.
 In order to use the independence of the increments of Brownian motion 
 we introduce a new process, independent of $\F_t$, which is an approximation of $X^1_{t+u}-X^1_t:$
\begin{equation}
\label{X1tu}
X^1_{t,u}:=A^1_k(X_t)
\hat W^k_u\mbox{ where }
\hat W^k_u:=W^k_{t+u}-W^k_t~;~
{ M_{t,h}:=\sup_{0\leq u\leq h}X^1_{t,u}}.
\end{equation}
{Lemma 
\ref{di-sup-cv-0} (ii) 
 will set}
${
\mathbb{E}\left[|M_{t+h}-M_t-
 (M_{t,h}-M_t+X_t^1)_+|
 \right]=o(h),}$
  where $(x)_+=\max(x,0)$.
Thus 
\begin{eqnarray}
\frac{1}{h}\mathbb{E}[\psi(V_t)(M_{t+h}-M_t)]= \mathbb{E}[\psi(V_t)(M_{t,h}-M_t+X_t^1)_+]+o(h)
\end{eqnarray}
Remark that the law of $M_{t,h}$  given  $\F_t$
is the law of $\|A^1(X_t)\| \sup_{0\leq u\leq h} \hat W^1_u,$  then { using the function $\h$ \eqref{def-hm},} a $\F_t$ conditioning yields:
\begin{eqnarray}
\frac{1}{h}\mathbb{E}[\psi(V_t)(M_{t+h}-M_t)]=\frac { 2}{\sqrt h}\,\mathbb{E}\left[\Psi({ V_t}){ \|A^1(X_t)\|}\h(\frac{M_t-X^1_t}{\sqrt h \|A^1(X_t)\|})\right]
+o(h).
\end{eqnarray}
Then $T_{h,t}=\frac { 2}{\sqrt h}\,\mathbb{E}\left[\Psi({ V_t}){ \|A^1(X_t)\|}\h(\frac{M_t-X^1_t}{\sqrt h \|A^1(X_t)\|})\right]
+o(h)$
as it { appears in} Proposition \ref{proDiagm}
{ (ii)}.
\\
In Proposition  \ref{pro-forme-dens-phi-t},
 under Hypothesis \ref{hyp-dens-outil}, 
 we compute $\lim_{h\to 0}T_{h;t}.$

{ Finally} in Section \ref{Sec3}  we prove 
$F_\psi: t\mapsto E[\int_0^t \psi(V_s)dM_s]$
is an absolutely continuous function with respect to Lebesgue measure, integral of its right derivative. 
Actually we prove that $F_\psi$ is a continuous function belonging to the Sobolev space $W^{1,1}(I),~I=(0,T).$ 
This achieves the proof of Theorem \ref{pro-sous-hyp-dens}.

The main propositions to prove are

\begin{pro}
  \label{proDiagm}
{Let  $B$ and $A$  fulfil  \eqref{h1h2} and \eqref{hyp:UnEl} {and
let} $\Psi \in C_b^1({\mathbb R}^{d+1},{\mathbb R})$. { Recall  {that}  }{$A^1$ {  is} the vector $(A^1_j,j=1,...,d),$   } and $\|A^{1}(x)\|^2=\sum_{j=1}^d (A^{1}_j(x))^2$. 
\\
(i) { for all $T>0$,} there exists a constant $C>0,$ { (depending on $\|A\|_{\infty},$ $\|B\|_{\infty}$,  $ \|\nabla A\|_{\infty},$ $\|\Psi\|_{\infty} ,$ $\|\nabla \Psi\|_{\infty}$ and $T$)} such that
$\forall t\in [0,T],$ $h\in [0,1],$
{\footnotesize
\begin{equation}
\label{diagonalem}
 \left|\mathbb{E}\left[\int_t^{t+h}\Psi(V_s)dM_s-{ 2}\sqrt h \left(\Psi({ V_t}){ \|A^1(X_t)\|}\h(\frac{M_t-X^1_t}{\sqrt h \|A^1(X_t)\|})\right)\right]\right|\leq C h { \|\nabla \Psi\|_\infty},
\end{equation}
}}
(ii)  for all $t>0,~h\in [0,1],$ 
 $$\lim_{h\to 0+}\frac{1}{h} \left|\mathbb{E}\left[\int_t^{t+h}\Psi(V_s)dM_s\right]-{ 2}\sqrt h\mathbb{E}\left[\Psi({ V_t}){ \|A^{1}({ X_t})\|}\h(\frac{M_t-X^1_t}{\sqrt h \|A^{1}({ X_t})\|})\right]\right|=0,$$
 where, denoting by   $\Phi_G$   the standard Gaussian cumulative distribution function,  
{\begin{align}
\label{def-hm}
\h(\theta):=\int_\theta^{\infty} \frac{1}{\sqrt{2\pi}} (y-\theta)e^{-\frac{y^2}{2}} dy=\frac{e^{-\frac{\theta^2}{2}} }{\sqrt{2\pi}}-\theta\Phi_{G}( -\theta).
\end{align}}
\end{pro}
The following remark { will be useful}:
 \begin{rem}
\label{remm}
The definition of $\h $ in  \eqref{def-hm}  implies that
 $\int_0^\infty \h(u)du=1/4$
 { Moreover, $\h'(\theta)=
 -\Phi_G(-\theta)\leq 0,$ {in particular}   $\h$ is non increasing.
}
\end{rem}

{ \begin{pro}
  \label{pro-forme-dens-phi-t}
	Assume that $A$ and $B$ fulfil  \eqref{h1h2} and \eqref{hyp:UnEl} and {that} $(M,X)$ fulfils Hypothesis \ref{hyp-dens-outil}, then for all $\Psi \in C^1_b({\mathbb R}^{d+1},{\mathbb R})$,
	for all $0 <T$  and {for all} $t\geq 0:$
	\begin{align*}
(i)	&t\mapsto \sup_{h >0}\frac{{ 2}\sqrt h}{h}\,\mathbb{E}\left[\Psi(V_t){ \|A^1(X_t)\|}\h(\frac{M_t-X^1_t}{\sqrt h \|A^1(X_t)\|})\right] \in L^1([0,T],{\mathbb R}),
	\\
	(ii)&
	\lim_{h \rightarrow 0^+} \frac{{ 2}\sqrt h}{h}\,\mathbb{E}\left[\Psi(V_t){ \|A^1(X_t)\|}\h(\frac{M_t-X^1_t}{\sqrt h \|A^1(X_t)\|})\right]
	\\
	&=\demi\int_{{\mathbb R}^d} \Psi(m,m,\tilde{x})\|A^1(m,\tilde{x})\|{ ^2}p_V(m,m,\tilde{x};t)dmd\tilde{x}
	\end{align*}
 As a corollary, {the  function
$t\rightarrow \demi\int_{{\mathbb R}^d} \Psi(m,m,\tilde{x}){\|A^1(m,\tilde{x})\|^2}p_V(m,m,\tilde{x};t)dmd\tilde{x}$} belongs to $ L^1([0,T],\mathbb{R}).
$
	\end{pro}}
	\noindent
	    The proof of Proposition \ref{proDiagm}
{will be} obtained  with the lemmas in the following section.

\subsection{Tools for proving Proposition \ref{proDiagm}}

Here we provide some estimations  of the expectations { of the increments of the
 processes $X$ and $M$}.
 { Assumptions  \eqref{h1h2} and \eqref{hyp:UnEl} allow us to introduce a constant  $K$ which denotes} either $ \max(\|A\|_{\infty},\|B\|_{\infty})$
or  $ \max(\|A\|_{\infty},\|B\|_{\infty},$
$\|\nabla A\|_\infty).$
Let $C_p$ be the   {constant in the   Burkholder-Davis-Gundy} inequality (cf. Theorem B.36 in \cite{BainCrisan}).

\begin{lem}
\label{maj-x-unif-h}
Let $A$ and $B$ be bounded. {Then,} {for all} $ 0<h\leq 1,$ {{for all} $ p\geq 1$} there exists a constant $C_{p,K}$
(depending only on $p$ and $K$) such that:
\begin{align*}
{ \sup_{t>0}}~\mathbb{E} \left[ \sup_{0\leq s \leq h}\| X_{t+s} -X_t\|^p \right] \leq 
C_{p,K} h^{p/2}
.
\end{align*}
\end{lem}
\begin{proof}
Using the fact that $(a+b)^p \leq 2^{p-1} \left[a^p +b^p \right],~a,b\geq 0$, one obtains:
\begin{align*}
0\leq \sup_{s \leq h} \| X_{t+s}-X_t \|^p \leq 2^{p-1}
\left[\sup_{u \leq h} \left( \|\int_t^{t+u} B(X_s) ds\| \right)^p+
 \sup_{u \leq h}\left( \|\int_t^{t+u} A_j(X_s) dW^j_s \|\right)^p  \right].
\end{align*}
{Taking  expectation of both members, {the  Burkholder-Davis-Gundy} inequality implies
$$\mathbb{E}[\sup_{s \leq h} \| X_{t+s}-X_t \|^p] \leq 2^{p-1}(1+C_p)
\mathbb{E}\left[ \left( \int_{t}^{t+h} \|B(X_s) \|ds \right)^p+\left( \int_t^{t+h} \|A(X_s)\|^2 ds\right)^{p/2}  \right].
$$
{ Assumption \eqref{h1h2} on} $B$ and $A$ yields
$\mathbb{E}[\sup_{s \leq h}{\| X_{t+s}-X_t \|^p}] \leq 2^{p-1}(1+C_p)(h^{p}K^p+ h^{p/2}
 K^p).
$
}
\end{proof}

\begin{lem}
\label{maj-sup-h}
{ Let  $B$ and $A$ satisfy Assumptions \eqref{h1h2} and \eqref{hyp:UnEl}.} Then, {for all} $ 0<h\leq 1,$ {for all} $ p\geq 1$ we get:
\begin{equation}
\label{majMinc}
{ \sup_{t>0}}~\mathbb{E}[| M_{t+h} -M_t|^p ]\leq
 C_{p,K}h^{p/2}~;~
\mathbb{E}[ | M_{t+h} -M_t|^p ]=o(h^{p/2}).
\end{equation}
\end{lem}

\begin{proof}
Recall
$
M_{t+h}-M_t= \left(  \sup_{0 \leq u \leq h} (X_{ t+u}^1 - X_t^1)+X_t^1 - M_t  \right)_+
$
recalling $(x)_+=\max(x,0).$
For any $a \geq 0,$  {one has} ${( x-a)_+} \leq |x|{{\mathbf 1}_{\{x >a\}}},$ thus
\begin{align*}
0 \leq M_{t+h}-M_t\leq |  \sup_{0 \leq u \leq h} (X_{ t+u}^1- X_t^1)| {\mathbf 1}_{ \{  \sup_{0 \leq u \leq h} (X_{ t+u}^1- X_t^1)> M_t-X_t^1\}}.
\end{align*}
 {Cauchy-Schwartz's} inequality yields:
\begin{align*}
0 \leq \mathbb{E} \left[ \left( M_{t+h}-M_t\right)^p \right]\leq
 \sqrt{ \mathbb{E}\left[|  \sup_{0 \leq u \leq h} (X_{ t+u}^1- X_t^1)|^{2p} \right] {\mathbb P} ( \{  \sup_{0 \leq u \leq h} (X_{t+u}^1- X_t^1)> M_t-X_t^1\})}.
\end{align*}
{ Replacing $p$ by $2p$ in
Lemma {\ref{maj-x-unif-h}} leads to the inequality in
 \eqref{majMinc} and {the equality}
 \\
 $\lim_{h\rightarrow 0}\sup_{0 \leq u \leq h} ({ X_{t+u}^1}- X_t^1)=0$ {holds almost surely.}
 According to Theorem 1.1  {in} \cite{LaureMo}  { extended to $X_0$ with law $\mu_0$
 on $\mathbb{R}^d$, }
the pair $(M_t,X_t)$ admits a density,
thus ${\mathbb P}  \{  M_t-X_t^1=0\}=0$ holds almost surely.
Therefore  $E \left( \left[ M_{t+h}-M_t\right]^p \right)$ is {  bounded}
by the  product of $h^{p/2}$ and a factor going 
to zero {when $h$ goes to 0,} and   {this quantity} is an $o(h^{p/2})$.}
\end{proof}

For any  fixed $t$ we { recall}  the process
{ $(X_{t,u},~~u \in [0,h])$}  and the running maximum of its first component {as follows}:
\begin{align}
\label{euler-approx}
 X_{t,u} := {  \sum_j A_j(X_t)} \hat W^j_u,~M_{t,h}: =\sup _{ 0 \leq u \leq h} X_{t,u}^1.
\end{align}
\begin{lem}
\label{lem-maj-dif-sup}
Under Assumptions \eqref{h1h2} and \eqref{hyp:UnEl}, {for all} {$  p\geq 1$ }
there exists a constant $C_{p,K}$ such that
 such that {for all} $  t\leq T$, {for all} $ h \in [0,1]$:
\begin{align*}
\mathbb{E} \left[ \sup_{s \leq h} | X^1_{s+t}-X^1_t - X^1_{t,s}|^p \right]\leq
C_{p,K}h^{p}.
\end{align*}

\end{lem}
\begin{proof}
By definition, { recalling $\hat{W}_u:= W_{t+u} - W_t,~~u \geq 0,$} we obtain
\begin{align*}
X^1_{s+t} -X^1_t - X^1_{t,s} = { \int_0^s} {  B^1(X_{u+t})}du +
{ \int_0^s}  \left[ A^1(X_{u+t}) - A^1(X_t)\right] d{ \hat{W}}_u.
\end{align*}
Using once again $(a+b)^p \leq 2^{p-1}(a^p+b^p ),~a,b\geq 0,$ {we get}
{\small
\begin{align*}
&\sup_{0\leq s \leq h} | X^1_{s+t}-X^1_t - X^1_{t,s}|^p \leq
\\
& 2^{p-1}
\left[\left( { \int_0^h}  { \|B^1(X_{u+t})\| } du\right)^p  +\sup_{0\leq s \leq h}
 \left\| { \int_0^s} \left[ A^1(X_{u+t}) - A^1(X_t)\right] d{ \hat{W}}_u\right\|^p \right].
 \end{align*}
 }
 Taking expectation of both sides and applying {the Burkholder-Davis-Gundy} inequality yield { with $D_p=2^{p-1}(1+ C_p)$}:
{\footnotesize
$$
\mathbb{E} \left[\sup_{0\leq s \leq h} | X^1_{s+t}-X^1_t - X^1_{t,s}|^p \right]
 \leq D_p
\left(\mathbb{E}\left[{ \int_0^h}
 { \|B^1(X_{u+t}) \| }du\right]^p
 +\mathbb{E}  \left| { \int_0^h}  \|A^1(X_{u+t}) - A^1(X_t)\|^2 du\right|^{p/2}\right). $$}
{ The first term above is bounded by $K^ph^p$ since $B$ is bounded. The assumption that $A$ belongs to $C^1_b(\mathbb{R}^d,\mathbb{R}^{d\times d})$
  and  Jensen's inequality imply {that} the second term is {bounded} by
  $K^ph^{p/2-1}{\int_0^h}E\|X_{u+t}-X_t\|^pdu$ thus}
$$
\mathbb{E}\left[\sup_{0\leq s \leq h} | X^1_{s+t}-X^1_t - X^1_{t,s}|^p \right]
 \leq D_p{  K^ph^{p/2-1}\left(h^{p/2+1} +
 { \int_0^h}\mathbb{E} \|X_{u+t}- X_t\|^pdu\right).}
$$
 From Lemma \ref{maj-x-unif-h} { we obtain the uniform upper bound:} 
 $\mathbb{E} [\|X_{u+t}- X_t\|^p]\leq 
  C_{p,K} u^{p/2}$ hence
$$
\mathbb{E} \left[\sup_{s \leq h}
 | X^1_{s+t}-X^1_t - X^1_{t,s}| ^p \right]
 \leq
\frac{D_pK^pC_{p,K}}{\frac{p}{2}+1}h^p.
$$
\end{proof}

\begin{lem}
\label{di-sup-cv-0}
Under Assumptions  { \eqref{h1h2} and (\ref{hyp:UnEl})}, one has
 \begin{align*}
&(i)~{ \exists C>0}~~\sup_{ 0 \leq t \leq T; ~~~0 \leq h \leq 1}  h^{-1}\mathbb{E}
 \left[\left| M_{t+h}- M_t - \left( M_{t,h} - M_t + X_t^1\right)_+ \right|\right] \leq C<\infty,
\\
&(ii)~~\lim_{h \rightarrow 0^+} h^{-1}\mathbb{E}\left[ \left| M_{t+h}- M_t -
\left( M_{t,h} - M_t + X_t^1\right)_+ \right|\right]=0.
\end{align*}
\end{lem}
\begin{proof} 

Fistly remark
\begin{align}
\label{maxp-plus-lip}
\forall a \in {\mathbb R},~~
\left| \left(x-a\right)_+ -\left(y-a\right)_+ \right| \leq \left|x-y\right|
\left[ {\mathbf 1}_{\{ x> a\}} + {\mathbf 1}_{\{ y> a\}} \right],
\end{align}
{and if $ f$ and $g$ are   functions on  $[0,T],$} then
\begin{align*}
\forall s\in[0,T],~~
f(s) -\sup_{0\leq u \leq T}g(u) \leq f(s) - g(s) \leq | f(s) - g(s)| \leq \sup_{v \leq T}| f(v) - g(v)|,
\end{align*}
 hence   $\sup_{s \leq T} f(s) - \sup_{u \leq T}g(u) \leq \sup_{v \leq T}|f(v) - g(v)|.$
Here the role of  $f$ and $g$ {is} symmetrical so
$
\sup_{s \leq T} g(s) - \sup_{u \leq T}f(u) \leq \sup_{v \leq T}| f(v) - g(v)|,
$
and
\begin{equation}
\label{sup-sup}
\left| \sup_{s \leq T} g(s) - \sup_{u \leq T}f(u) \right| \leq \sup_{v \leq T}|f(v) - g(v)|.
\end{equation}
We now consider 
$
M_{t+h}-M_t =\left( \sup_{0 \leq u \leq h}(X_{u+t}^1 -X_t^1) -M_t + X_t^1\right)_+, 
$
using \eqref{maxp-plus-lip}
\begin{align*}
&\left| M_{t+h}- M_t - \left( M_{t,h} - M_t + X_t^1\right)_+ \right|
\leq
\\
& \left|  \sup_{0 \leq u \leq h}(X_{u+t}^1 -X_t^1) - M_{t,h}
  \right| \left[ {\mathbf 1}_{\{  \sup_{0 \leq u \leq h}(X_{u+t}^1 -X_t^1) > M_t-{ X^1_t}\}} +
  {\mathbf 1}_{\{ M_{t,h}> M_t-X_t^1\}} \right].
\end{align*}
Then, for  {any} $t$ fixed, we apply  inequality \eqref{sup-sup}
  to the maps $g:u\mapsto X_{{u+t}}^1 {-X_t}^1$ and 
  \\
  $f:u\mapsto X_{t,u}^1$. Then
  $\left| M_{t+h}- M_t - \left( M_{t,h} - M_t + X_t^1\right)_+ \right| \leq$
\begin{align*}
 \sup_{0\leq u \leq h} \left|X_{u+t}^1 -X_t^1 - X_{t,u}^1\right|
  \left[ {\mathbf 1}_{\{  \sup_{0 \leq u \leq h}(X_{u+t}^1 -X_t^1)  > M_t-X_t^1\}} +
  {\mathbf 1}_{\{ M_{t,h}> M_t-X_t^1\}} \right].
\end{align*}
From Cauchy-Schwartz's inequality and the fact that $(a+b)^2 \leq 2(a^2+b^2),$  we get
\begin{align*}
&\mathbb{E} \left[ \left| M_{t+h}- M_t - \left( M_{t,h} - M_t + X_t^1\right)_+ \right|\right]\leq
\\
&\sqrt{2\mathbb{E}\left[\sup_{u \leq h}
 \left|X_{u+t}^1 -X_t^1 - X_{t,u}^1\right| ^2\right]
\left( {\mathbb P} \{  \sup_{0 \leq u \leq h}
 (X_{u+t}^1 -X_t^1) > M_t-X_t^1\}+
{\mathbb P} \{ M_{t,h}> M_t-X_t^1\} \right)}.
\end{align*}
Lemma \ref{lem-maj-dif-sup} with $p=2$ insures that {the map} $h \mapsto h^{-1}\sqrt{2\mathbb{E}
 \left[  \sup_{u \leq h} \left|X_{u+t}^1-X_t^1 - X_{t,u}^1\right| ^2\right] }$
  is uniformly bounded  in $t.$
Concerning the second factor,
\\
$\bullet$  firstly  the almost sure continuity with respect to $h$  insures that the {quantities}
\\
$\lim_{h\rightarrow 0}  \sup_{0 \leq u \leq h}(X_{u+t}^1 -X_t^1) $ and $\lim_{h\rightarrow 0} M_{t,h}$
 are equal to $0;$
\\
$\bullet$ secondly the {law of the} pair $(M_t,X_t)$  admits a density  { with respect to the Lebesgue measure on $\bar\Delta$} according to Theorem 1.1 \cite{LaureMo} so
 {${\mathbb P}( \{  0= M_t-X_t^1\})=0$} and  the limit of the second factor  is equal~to~$0.$
\\
 This concludes the proof of the lemma.
\end{proof}
\noindent
Recall Definition \eqref{euler-approx}:
$
X_{t,h}= A_j(X_t)[ W_{t+h}^j-W_t^j],~{ M_{t,h}=\mathbf{\sup_{0\leq u\leq h}} {X_{t,u}^1},}
~h \in [0,1].
$
\begin{lem}
{ Under Assumptions  \eqref{h1h2} and (\ref{hyp:UnEl}), with  $\h$  defined in \eqref{def-hm}:}
\label{pro:calcul-esp-cond}
$$
\mathbb{E}\left[  ( M_{t,h} - M_t + X_t^1 )_+ | {\mathcal F}_t\right] =2\|A^1(X_t)\| \sqrt{h}\h\left(\frac{M_t -X_t^1}{\|A^{1}(X_t)\|\sqrt{h}}\right).
$$
\end{lem}
\begin{proof}
{For any $t$ fixed,} conditionally to ${\mathcal F}_t$ 
 the process {$(X_{t,u}^1, ~u \in [0,h])$ \eqref{X1tu} has the same law
as $(\sqrt{h}\|A^{1}(X_t)\|\hat W_u,~u\in [0,1])$ where $\hat W$ is a Brownian motion independent of ${\mathcal F}_t$,
and  {for any $h$,} the random variable $M_{t,h}$ has the same law as $ \sqrt{h}\|A^{1}(X_t)\|\sup_{u \leq 1}\hat{ W}_u.$}
\\
Following { \cite{CJY} Section 3.1.3, } {the random variable} $\sup_{u \leq 1}\hat W_u$ has the same law as $|G|$ where $G$ is a standard Gaussian variable (independent of ${\mathcal F}_t),$ with   density $\frac{2}{\sqrt{2\pi}}e^{-\frac{z^2}{2}}{\mathbf 1}_{[0,+\infty [}(z)$.
Then { using the function $\h$  introduced in} \eqref{def-hm}
\begin{align*}
\mathbb{E}\left[(M_{t,h} -(M_t -X_t^1))_+ |{\mathcal F}_t \right]
&= \int_0^{\infty} \left( \|A^{1}(X_t)\|\sqrt{h} z - (M_t -X_t^1)\right)_+ \frac{2}{\sqrt{2\pi}}e^{-\frac{z^2}{2}}dz
\\
&= 2 \|A^{1}(X_t)\|\sqrt{h} \h( \frac{M_t -X_t^1}{\sqrt{h}\|A^{1}(X_t)\|}).
\end{align*}
\end{proof}

\subsection{Proof of Proposition \ref{proDiagm} }
\label{Sec2.3}

{ Let $t>0$.} The key of this proof is to write the quantity
$$
\mathbb{E}\left[\int_t^{t+h}\Psi(V_s)dM_s\right]-{ 2}\sqrt h\mathbb{E}\left[\Psi({ V_t}){ \|A^{1}({ X_t})\|}\h(\frac{M_t-X_t^1}{\sqrt h \|A^{1}({ X_t})\|})\right]
$$
as the sum of { three}  terms,
\begin{align}
\label{threeterms}
&\mathbb{E}\Big[\int_t^{t+h}(\Psi(V_s)-\Psi(V_t))dM_s\Big]+
{ \mathbb{E}\Big[\Psi(V_t)\Big((M_{t+h}-M_t)- \mathbb{E}\left[M_{t,h}-M_t+X_t^1)_+|  {\mathcal F}_t \right]\Big)\Big]}
\\
&+\mathbb{E}\left[\Psi(V_t) { \mathbb{E}} \left[(M_{t,h}-M_t+X_t^1)_+\left| \right. {\mathcal F}_{ t} \right]
-{ 2}\sqrt h\Psi(V_t){ \|A^{1}(X_t)\|}\h(\frac{M_t-X_t^1}{\sqrt h \|A^{1}(X_t)\|})\right].\nonumber
\end{align}
{We now prove that each terms in sum \eqref{threeterms} are both $o(h)$ and $O(h)$ uniformly in time.}
\\
(a) Using Lemma \ref{pro:calcul-esp-cond} the { third} term is null.
\\
 (b)  Concerning the second term, using the fact that $\Psi$ is bounded  and 
  { Lemma \ref{di-sup-cv-0} (i) {for all} $  t\in [0,T]$
\begin{align*}
&\left|\mathbb{E}\left[\Psi(V_t)[(M_{t+h}-M_t)- \mathbb{E}[(M_{t,h}-M_t+X_t^1)_+\left| \right. {\mathcal F}_{ t} ]\right]\right|\leq
\\
&\|\Psi\|_\infty\left|\mathbb{E}\left[M_{t+h}-M_t-
 \mathbb{E}[(M_{t,h}-M_t+X_t^1)_+| {\mathcal F}_t]\right]\right|{ \leq Ch\|\Psi\|_\infty},
\end{align*}
{ as it is required in \eqref{diagonalem}}.
{Moreover using  Lemma \ref{di-sup-cv-0} (ii)
\begin{align*}
&\lim_{h\to 0}\frac{1}{h}\left|\mathbb{E}\left[\Psi(V_t)[(M_{t+h}-M_t)- \mathbb{E}[(M_{t,h}-M_t+X_t^1)_+\left| \right. {\mathcal F}_{ t} ]\right]\right|=0.
\end{align*}}
\\
(c) Since $\nabla \Psi$ is bounded and the process $M$
is increasing, the first term is {  bounded:}
$$
\mathbb{E}\left[\int_t^{t+h}[\Psi(V_s)-\Psi(V_t)]dM_s\right]\leq
\|\nabla\Psi\|_\infty \mathbb{E}[\sup_{t\leq s\leq t+h}
\|V_s-V_t\|(M_{t+h}-M_t)].
$$
Using {Cauchy-Schwarz's} inequality
$$\mathbb{E}\left[\sup_{t\leq s\leq t+h}\|V_s-V_t\|(M_{t+h}-M_t)\right]
\leq
 \sqrt{\mathbb{E}[\sup_{t\leq s\leq t+h}\|V_s-V_t\|^2]\mathbb{E}[(M_{t+h}-M_t)^2]}.
$$
{ Since  $\|V_s-V_t\|^2=(M_s-M_t)^2+
\|X_s-X_t\|^2,$ we {obtain}
\\
$\sup_{t\leq s\leq t+h}\|V_s-V_t\|^2\leq
(M_{t+h}-M_t)^2+\sup_{t\leq s\leq t+h}\|X_s-X_t\|^2,$ hence}
{\footnotesize
$$ \mathbb{E}[\sup_{t\leq s\leq t+h}\|V_s-V_t\|(M_{t+h}-M_t)]
\leq\sqrt{ \mathbb{E}[(M_{t+h}-M_t)^2]+ \mathbb{E}[\sup_{t\leq s\leq t+h}\|X_s-X_t\|]^2)}\sqrt{\mathbb{E}[(M_{t+h}-M_t)^2]}
$$
}
Lemmas \ref{maj-x-unif-h} and  \ref{maj-sup-h}
($p=2$)
yield { the fact that  the first   factor is an $o(\sqrt h)$ and the second  is an $O(\sqrt h)$  uniformly with respect to {$t\geq 0$}. Then $\mathbb{E}[\sup_{t\leq s\leq t+h}\|V_s-V_t\|(M_{t+h}-M_t)]$
is an $o(h)$ and an $O(h)$ 
uniformly with respect to {$t\geq 0$}.}
\hfill$\Box$

{
\subsection{Proof of Proposition \ref{pro-forme-dens-phi-t}}
\label{Sec2.4}

{ (i) Recall}   that $A$ and $B$ fulfil  \eqref{h1h2}, \eqref{hyp:UnEl} and $(M,X)$ fulfils Hypothesis \ref{hyp-dens-outil}.
Then, using {the density $p_V$  
  of
the {law} of the   {pair} $(M_t,X_t)$
%
 we have
\begin{align*}
&{\mathbb E} \left[\Psi(V_t) \|A^1(X_t)\| {\mathcal H} \left( \frac{M_t-{ X^1_t}}{\sqrt{h}\|A^1(X_t)\|}\right) \right]\leq
\\
&\|\Psi\|_{\infty} \|A\|_{\infty}\int_{{\mathbb R}^{d+1}}{\mathcal H}\left( \frac{m-x^1}{\sqrt{h}\|A^1(x^1,\tilde{x})\|}\right) p_V(m,x^1,\tilde{x};t) dm~ dx^1~d\tilde{x}.
\end{align*}
The change of variable { $x^1=m-u\sqrt{h}$} yields
\begin{align}
\label{rep-dens-t-1}
&\frac{\sqrt{h}}{h}{\mathbb E} \left[ \Psi(V_t)\|A^1(X_t)\| {\mathcal H} \left( \frac{M_t-{ X^1_t}}{\sqrt{h}\|A^1(X_t)\|}\right) \right]\leq \\
&\|\Psi\|_{\infty} \|A^1\|_{\infty}\int_{{\mathbb R}^d\times[0,+\infty[} {\mathcal H}\left( \frac{u}{\|A^1(m-\sqrt{h}u,\tilde{x})\|}\right) p_V(m,m-\sqrt{h}u,\tilde{x};t) dm~ d\tilde{x}~du.\nonumber
\end{align}
\\
Since $\h$ is decreasing (Remark \ref{remm})
 and  $0\leq h\leq 1,$
 $\h\left( \frac{u}{\|A^1(m-\sqrt{h}u,\tilde{x})\|}\right)\leq \h(\frac{u}{\|A^1\|_\infty}):$ 
 \begin{align*}
&\left|\frac{\sqrt{h}}{h}{\mathbb E} \left[\Psi(V_t) \|A^1(X_t)\| {\mathcal H} \left( \frac{M_t-{X^1_t}}{\sqrt{h}\|A^1(X_t)\|}\right) \right]\right|\leq
\\
&\|\Psi\|_{\infty} \|A^1\|_{\infty}
\int_{{\mathbb R}^d\times[0,+\infty[}\h\left(\frac{u}{\|A^1\|_\infty}\right)\sup_{r>0}p_V(m,m-r,\tilde{x};t) dm~ d\tilde{x}~du.
\end{align*}
%

 { Applying Tonelli's} Theorem, computing the integral with respect to $du$ in  the right-hand with}  $\int_0^{\infty}\h(v)dv=1/4$ { (Remark \ref{remm})}, yield:
 {\footnotesize
 \begin{align*}
\sup_{h>0}\left|\frac{\sqrt{h}}{h}{\mathbb E} \left[\Psi(V_t) \|A^1(X_t)\| {\mathcal H} \left( \frac{M_t-{ X^1_t}}{\sqrt{h}\|A^1(X_t)\|}\right) \right]\right|\leq
\frac{1}{4} \|\Psi\|_{\infty} \|A^1\|^2_{\infty}
\int_{{\mathbb R}^{d}}\sup_{r>0}p_V(m,m-r,\tilde{x};t) dm~ d\tilde{x} .
\end{align*}
}

{Using Hypothesis} \ref{hyp-dens-outil} (i), {we obtain that the map}:
$$
t\mapsto \sup_{h>0}\left|\frac{\sqrt{h}}{h}{\mathbb E} \left[\Psi(V_t) \|A^1(X_t)\| {\mathcal H} \left( \frac{M_t-X^1_t}{\sqrt{h}\|A^1(X_t)\|}\right) \right]\right|
$$
  belongs to $  L^1([0,T], \mathbb{R})$.
Point (i) of Proposition \ref{pro-forme-dens-phi-t} is proved.

{(ii) Concerning the proof of point (ii),} firstly note that
{\footnotesize
\begin{align*}
{\mathbb E} \left[ \Psi(V_t)\|A^1(X_t)\| {\mathcal H} \left( \frac{M_t-X^1_t}{\sqrt{h}\|A^1(X_t)\|}\right)\right]=
\int_{{\mathbb R}^{d+1}} \Psi(m,x) \|A^1(x)\| {\mathcal H} \left( \frac{m-x^1}{\sqrt{h}\|A^1(x)\|}\right) p_V(m,x;t) dm ~dx.
\end{align*}
}
{After the } change of variable { $x^1=m-u\sqrt{h},$} {we obtain}
{\footnotesize
\begin{align}
\label{rep-dens-t-1-bis}
&\frac{\sqrt{h}}{h}{\mathbb E} \left[ \Psi(V_t)\|A^1(X_t)\| {\mathcal H} \left( \frac{M_t-{ X^1_t}}{\sqrt{h}\|A^1(X_t)\|}\right) \right] =
\\
&\int_{{\mathbb R}^d\times\mathbb R^+} \Psi(m,m-u\sqrt{h},\tilde{x})\|A^1(m-u\sqrt{h},\tilde{x})\|{\mathcal H}\left( \frac{u}{\|A^1(m-\sqrt{h}u,\tilde{x})\|}\right) p_V(m,m-\sqrt{h}u,\tilde{x};t) dm~ d\tilde{x}~du.\nonumber
\end{align}}
 Using {Lebesgue's} dominated convergence Theorem, we  { let $h$ go} to $0$ in \eqref{rep-dens-t-1-bis}
{ for $t>0$}, {and} using the fact that { $\Psi,$} $A$ and ${\mathcal H}$ are continuous and Hypothesis \ref{hyp-dens-outil} (ii) we obtain
\begin{align*}
&\lim_{h \rightarrow 0}\frac{\sqrt{h}}{h}{\mathbb E} \left[ \Psi(V_t)\|A^1(X_t)\| {\mathcal H} \left( \frac{M_t-X_t}{\sqrt{h}\|A^1(X_t)\|}\right) \right]=
\\
&\int_{{\mathbb R}^{d}\times[0,+\infty[}\Psi(m,{m,\tilde{x}})\|A^1(m,\tilde{x})\| \h\left( \frac{u}{\| A^1({ m,\tilde{x}})\|}\right) {p_V(m,m,\tilde{x};t)}dm~d\tilde{x} ~du.
\end{align*}
Using   the change of variable $z=\frac{u}{\|A^1(m,\tilde{x})\|},$  and Remark \ref{remm} $\int_0^{\infty}\h(z)dz=1/4,$  yields
{\footnotesize
\begin{align*}
\lim_{h \rightarrow 0}\frac{\sqrt{h}}{h}{\mathbb E} \left[ \Psi(V_t)\|A^1(X_t)\| {\mathcal H} \left( \frac{M_t-X^1_t}{\sqrt{h}\|A^1(X_t)\|}\right) \right]=\frac{1}{4}\int_{{\mathbb R}^{d}}\Psi(m,{m,\tilde x})\|A^1(m,\tilde{x})\|^2 p_V(m,{ m,}\tilde{x};t)dm~d\tilde{x} .
\end{align*}
}
\hfill$\Box$

\subsection{End of proof of Theorem \ref{pro-sous-hyp-dens}}
\label{Sec3}

 We recall Theorem 8.2 page 204 in  Brezis
\cite{brezis}: let $f\in W^{1,1}(0,T)$,
then $f$ is almost surely equal to an absolutely continuous function. As a particular case,
any $f\in W^{1,1}(0,T)\cap C(0,T)$
is  absolutely continuous.
Recall $F_{\psi}~:t \mapsto  {{\mathbb E}\left[\int_0 ^t \Psi(V_s)dM_s \right]}.$}

\begin{lem}
\label{lem-cont-F-psi}
Assume that $A$ and $B$ fulfil \eqref{h1h2} and \eqref{hyp:UnEl} and {that} $\Psi$ is a continuous bounded function. Then $F_{\Psi}$   is a continuous function on $\mathbb{R}^+.$ 
\end{lem}
\begin{proof} Let $0 \leq s\leq t.$ Since $\Psi$ is bounded and $M$ is non decreasing
\begin{align*}
\left| F_{\Psi}(t)-F_{\Psi}(s)\right|=\left|{\mathbb E}\left[ \int_s^t \Psi(V_u) dM_u\right]\right|\leq \|\Psi\|_{\infty} {\mathbb E} [M_t-M_s].
\end{align*}
{The {map}  $t\mapsto {\mathbb E}[M_t]$ {being} continuous,
$F_{\Psi}$ is a continuous function on $\mathbb{R}^+.$}
\end{proof}

{
\begin{lem}
\label{lem-F-psi-W11}
Assume that $A$ and $B$ fulfil \eqref{h1h2} and \eqref{hyp:UnEl}, $(M,X)$ fulfils  Hypothesis \ref{hyp-dens-outil} and $\Psi \in C^1_b.$ Then  for all $T>0$, {the map}
$F_{\psi}$ belongs to the Sobolev space $W^{1,1}(]0,T[)$
and its weak derivative is
\begin{align*}
\dot{F}_{\Psi}(t):=
\frac{1}{2}\int_{{\mathbb R}^{d}}\Psi(m,\tilde{x})\|A^1(m,\tilde{x})\|^2 p_V(m,m,\tilde{x};t)dmd\tilde{x}
\end{align*}
\end{lem}
\begin{proof}
Let $g~: [0,T]\rightarrow \mathbb{R}$ be $C^1$ with compact support $[\alpha,\beta]\subset(0,T).$
  This means both functions $g$ and $\dot g$ are continuous so bounded and that moreover { $g(\alpha)=g(\beta)=0.$}
Note that
$
\dot{g}(t)= \lim_{h \rightarrow 0} \frac{g(t) -g(t-h)}{h},~~\forall t \in (0,T)
.$
Moreover, 
\\
 $\sup_{t \in [0,T]}\sup_{h \in [0,1]}|\frac{g(t) -g(t-h)}{h}| \leq \|\dot{g}\|_{\infty}.$
Observe that, since $M$ is non decreasing and  the coefficients $A$ and $B$ are bounded
$
\left| F_{\psi}(t)\right|\leq \|\Psi\|_{\infty} {\mathbb E} [M_T] <\infty.
$

Then, using  {Lebesgue's dominated convergence} Theorem
\begin{align*}
\int_0^T \dot{g}(s) F_{\psi}(s)ds=\int_0^T\lim_{h \rightarrow 0} \frac{g(s) -g(s-h)}{h}F_{\psi}(s)ds=\lim_{h \rightarrow 0}\int_0^T \frac{g(s) -g(s-h)}{h}F_{\psi}(s)ds.
\end{align*}

Using {the} change of variable $u=s-h$ in the last integral
\begin{align*}
\int_0^T &\frac{g(s) -g(s-h)}{h}F_{\Psi}(s)ds=h^{-1}\int_0^T g(s) F_{\Psi}(s) ds -h^{-1}\int_{-h}^{T-h} g(u) F_{\Psi}(u+h)du
\\
&= \int_{0}^T g(s) \frac{F_{\Psi}(s) -F_{\Psi}(s+h)}{h}ds{-}h^{-1}\int_{-h}^0 g(s)F_{\Psi}(s+h) ds {+}h^{-1}\int_{T-h}^Tg(s) F_{\Psi}(s+h) ds.
\end{align*}
{Recalling $supp(g)=[\alpha,\beta]\subset(0,T)$,
 $gF_{\Psi}$ is bounded on $[0,T]$ 
{ extended by $0$ on  $[\alpha,\beta]^c$ }
so
  $\lim_{s\rightarrow 0} g(s)=\lim_{s\rightarrow T} g(s)=0$
  then 
  { $h^{-1}\int_{-h}^0 g(s)F_{\Psi}(s+h) ds= h^{-1}\int_{T-h}^Tg(s) F_{\psi}(s+h) ds=0$  as soon as $0<h\leq T-\beta$}
  thus
$\lim_{h \rightarrow 0} \left[h^{-1}\int_{-h}^0 g(s)F_{\Psi}(s+h) ds\right] =$
\\
$\lim_{h \rightarrow 0}\left[ h^{-1}\int_{T-h}^Tg(s) F_{\psi}(s+h) ds\right]=0
$
 Applying  {Lebesgue's dominated convergence}
 Theorem yields, $ F$ admits   a weak derivative:
\begin{align*}
\int_0^T \dot{g}(s) F_{\psi}(s)ds=-\int_0^T g(s) \dot{F}_{\Psi}(s) ds.
\end{align*}
}
 {
Using Proposition \ref{proDiagm} (ii)
{
$$
\lim_{h\to 0^+}\left(-\frac{F_{\Psi}(t) -F_{\Psi}(t+h)}{h}- \frac{2}{\sqrt h}{\mathbb E}
\left[ \Psi(V_t)\|A^1(X_t)\|{\mathcal H}\left( \frac{M_t-X^1_t}{\sqrt{h}\|A^1(X_t)\|}\right)\right]\right)=0.
$$}
}
Using Proposition \ref{pro-forme-dens-phi-t}
(ii):
\begin{align*}
-\dot{F}_{\Psi}{(t+)}:=\lim_{h \rightarrow 0,h>0}\frac{F_{\Psi}(t) -F_{\Psi}(t+h)}{h}
=-\frac{1}{2}\int_{{\mathbb R}^{d}}\Psi(m,m,\tilde{x})\|A^1(m,\tilde{x})\|^2 p_V(m,m,\tilde{x};t)dmd\tilde{x} ,
\end{align*}
and the 
 points (i) of Propositions \ref{proDiagm} and  \ref{pro-forme-dens-phi-t}:
\begin{align*}
\sup_{h>0} \left| \frac{F_{\Psi}(t) -F_{\Psi}(t+h)}{h}\right| \in L^1([0,T],dt),
\end{align*}
so $\dot{F}_{\Psi}\in L^1([0,T],\mathbb{R}).$ 
\\
According to \cite{brezis} Chap 8 section 2 page 202, $F_{\Psi}$ belongs to $W^{1,1}(]0,T[,{\mathbb R}).$
\end{proof}}

We now end the proof of Theorem \ref{pro-sous-hyp-dens}:
According to Theorem 8.2  page 204 of \cite{brezis}, $F_{\psi}$ is equal almost surely to an absolutely continuous function.
Since $F_{\Psi}$ is continuous {(Lemma \ref{lem-cont-F-psi}),} the equality holds everywhere. Then $F_{\Psi}$ is an absolutely continuous function and its derivative is its right derivative.
\hfill$\Box$


\section{Case $A=I_d$}
\label{density}

In this rather technical section, we firstly  prove that the {density of the pair } $(M_t,X_t)$ { fulfils Hypothesis \ref{hyp-dens-outil}: $p_V$ \eqref{(*)} is continuous on the boundary of $\bar{\Delta}$ and}  is dominated by an integrable function:
\begin{pro}
\label{cont-ui--dens}
 Assume  that $B$ fulfils { Assumption \eqref{h1h2}} and $A=I_d,$
 then {$(M,X)$ fulfils Hypothesis \ref{hyp-dens-outil} meaning that for all { probability measure $\mu_0$} on ${\mathbb R}^d$ }
\begin{align*}
(i)&~~\forall T>0,~~\sup_{(h,u) \in [0,1]\times {\mathbb R}_+}p_V(b,b-hu,\tilde{a};t,{ \mu_0}) \in L^1([0,T] \times {\mathbb R}^d, dtdbd\tilde{a}).
\\
(ii)& \mbox{ Almost surely in } (m, \tilde{x}) \in {\mathbb R}^d, \forall t>0,
~~\lim_{u \rightarrow 0,u>0} p_V(m,m-u,\tilde{x};t,{ \mu_0})=p_V(m,m,\tilde{x};t,{ \mu_0}).
\end{align*}
\end{pro}
	{ As a by product
	{ using Theorems \ref{thme-sous-hyp-dens} and \ref{pro-sous-hyp-dens}}
	 this proposition
	achieves the proof of Theorem \ref{existence-dens}.
The main tool for the proof of this proposition is an integral representation of the density:
\begin{pro}
\label{rem-dens-mal}
 For {any  probability measure $\mu_0$} on ${\mathbb R}^d,$
for all $t>0$,
\begin{equation}
\label{rep-p_v}
p_V=p^0-\sum_{k=m,1,\cdots, d}
 (p^{k,\alpha}+   p^{k,\beta})
\end{equation}
where  the various $p$ are defined by
($\partial_k$ is the derivative with respect to  $k= m, x^1,...,x^d$ and $B^m=B^1$):
\begin{align*}
 &p^0(m,x;t):={ \int_{{\mathbb R}^d}p_{W^{*1},W}(m-x_0^1,x-x_0;t)\mu_0(dx_0)},
  \\
&{ p^{k,\alpha}}(m,x;t):=\int_0^t{\int_{\mathbb{R}^{d+1}} {\mathbf 1}_{b <m}B^{k}(a)\partial_{k}p_{W^{*1},W}\left(m-a^1,x-a;t-s\right)p_V(b,a;s)dbda}ds,
\\
&{ p^{k,\beta}}(m,x;t):=\int_0^t \int_{{\mathbb R}^{(d+1)}}
{\mathbf 1}_{b<m}
 B^{k}(a)\partial_{k}p_{W^{*1},W}(b-a^1,x-a;t-s)p_V(m,a;s)
dbda ds,
\end{align*}
\end{pro}
{ where $p_{W^{*1},W}(.,.;t)$ is the density of the distribution of $(\sup_{s\leq t}W^1_s,W_t)$ for $t \geq 0$, see Appendix \ref{brownian}}


\subsection{Integral representation of the density: proof of Proposition \ref{rem-dens-mal}}
\label{sec-dens-maj}

{Let $t >0$ be fixed}. { Firstly, we assume that $\mu_0=\delta_{x_0},~~x_0$ being fixed in ${\mathbb R}^d.$}
\\
{ According to Lemma \ref{lem-cont-al} below} and using the fact that $B$ is bounded,
$\forall t\in [0,T]$, 
 the functions 
  $ p^{k,\gamma}\in^\infty L\left([0,T], L^1(\mathbb{R}^{d+1})\right)$ 
 for $\gamma=\alpha,\beta$ .
\\
Let $F\in C^1_b({\mathbb R}^{d+1},{\mathbb R})$ with compact support. We will prove 
\begin{equation}
\label{devF(M,B)}
\mathbb{E}_\mathbb{P}[F(M_t,X_t)]=
\int_{\mathbb{R}^{d+1}}F(m,x)\left(p^0-\sum_{k=m,x^1,...,x^d}(p^{k,\alpha}+p^{k,\beta})(m,x,t)\right)dmdx.
\end{equation}
Using Malliavin calculus we obtain the following decomposition:
\begin{lem}
\label{firststep}
{\footnotesize
\begin{align*}
{\mathbb E}_\mathbb{P} \left[F(M_t,X_t)\right]
&= \int_{{\mathbb R}^{d+1}}  F(x_0^1+b,x_0+a)p_{W^{*1},W}(b,a;t)dbda\nonumber
\\
&+  \int_0^t {\mathbb E}_\mathbb{P}
\left[\int_{{\mathbb R}^{d+1}}
\partial_mF \left(X^1_s + b,X_s+ a\right)
{\mathbf 1}_{\{ M_s <X^1_s + b\}}  B^1(X_s)p_{W^{*1},W}(b,a;t-s)dbda
 \right]ds
\nonumber
\\
&+   \int_0^t {\mathbb E}_\mathbb{P}\left[\int_{{\mathbb R}^{d+1}}
{ \partial_{k}}  F\left(\max \left( M_s, X_s^1 + b\right),X_s+ a\right)
 B^k(X_s)p_{W^{*1},W}(b,a;t-s)dbda \right]ds.
\end{align*}
}
\end{lem}
\begin{proof}
Let $Z$ be the exponential martingale 
solution of
\begin{align}
\label{rep-dens-z}
Z_t& =
1+ 
\int_0^t Z_s B^k(x_0+W_s)dW^k_s.
\end{align}
{ As previously Einstein's convention is used.}
{Let ${\mathbb Q}=Z{\mathbb P},$
according to {Girsanov's} Theorem,
 using \eqref{rep-dens-z},
  $W_B:=\left(W_t^k-\int_0^t B^k(x_0+W_s) ds; k=1,...,d\right)_{t\geq 0}$ 
  is a ${\mathbb Q}$ continuous martingale such that $\langle W_B\rangle_t=t,~~\forall t \geq 0.$ That means that under ${\mathbb Q},$ $W_B$ is a $d$-dimensional Brownian motion. Then the distribution of $X$ (resp. $(M,X))$ under ${\mathbb P}$ is the distribution of
 $W+x_0$ (resp. $(W^{1*}+x_0,W+x_0))$ under ${\mathbb Q}$ and }
\begin{align}
\label{girsanov}
&{\mathbb E}_\mathbb{P}\left [ F(M_t,X_t)\right]=
{\mathbb E}_\mathbb{Q} \left [ F(x_0^1+W^{1*}_t,x_0+W_t)\right]= {\mathbb E}_\mathbb{P} \left [ F(x_0^1+W^{1*}_t,x_0+W_t)Z_t\right].
\end{align}
Let  $G:= F(x_0^1+W^{1*}_t,x_0+W_t)$ and $u:=ZB(x_0+W)$, using \eqref{rep-dens-z}: 
{
\begin{align}
\label{rep-dens-2}
\mathbb{E}_\mathbb{P}[F(M_t,X_t)]={\mathbb E}_\mathbb{P} \left [ F(x_0^1+W^{1*}_t,x_0+W_t)\right]+
{\mathbb E}_\mathbb{P} \left [G\delta(u) \right]
\end{align}
}

{ As a first step we will apply 
\eqref{prop-nualart(ii} (Appendix)
to the second term in \eqref{rep-dens-2}. Thus
we have to check that the pair $(G,u)\in \DD^{1,2}\times {\mathbb L}^{1,2}$:
{ $F$  being bounded and smooth, $G\in \DD^{1,2}$; }
%
and according to Lemma \ref{lem-z-l12}, the process $u$ belongs to ${\mathbb L}^{1,2}$.
\\
 Using \eqref{der-sup} (${\tau:=\inf\{s, W_s^{1*} =W^{1*}_t\}}$)
the pair  $(W^{1*}_t,W_t)$ belongs to ${\mathbb D}^{1,2}$
with Malliavin gradient:
\begin{align*}
&D_sW^{1*}_t=\left({\mathbf 1}_{[0,\tau]}(s),0,.....,0\right),~
D_sW_t^k=(\delta_{j=k},~~j=1,...,d){\mathbf 1}_{[0,t]}(s),~~k=1,...,d.
\end{align*}
%
}
Using the chain rule:
\begin{align*}
 \langle DG,u\rangle_{{\mathbb H}}&=
 \int_0^{t}\partial_m F({ x^1_0}+W^{*1}_t,x_0+W_t) {\mathbf 1}_{\{W^{1*}_s<W^{1*}_t\}}B^1(x_0+W_s){ Z_s} ds
 \\
 &+ \int_0^{t}{ \partial_{k}}  F({ x^1_0}+W^{*1}_t,x_0+W_t)B^k(x_0+W_s){ Z_s}ds.
\end{align*}
%
We are now in position to apply  \eqref{prop-nualart(ii} $E_\mathbb{P}[G\delta(u)]=E_\mathbb{P}
[\langle DG,u\rangle_\HH]$:
{
\begin{eqnarray}
\label{rep-dens-3}
& & \mathbb{E}_\mathbb{P}[G\delta(u)]=
\mathbb{E}_\mathbb{P}\left[\int_0^t
   \partial_m F (x_0^1+W^{1*}_t,x_0+W_t){\mathbf 1}_{\{ W^{1*}_s <W^{1*}_t \}}
   B^1(x_0+W_s)Z_sds\right]\nonumber
   \\
   & + &\mathbb{E}_\mathbb{P}\left[\int_0^t
  { \partial_{k}} F (x_0^1+
    W^{1*}_t,x_0+W_t)B^k(x_0+W_s)
    Z_s ds\right].
\end{eqnarray}
}
Plugging identity \eqref{rep-dens-3} into
{right hand of \eqref{rep-dens-2} and using Fubini Theorem to commute the integrals in $ds$ and $d\mathbb{P}$}, we obtain
{
\begin{eqnarray}
\label{34-32}
&&{\mathbb E}_\mathbb{P} \left [F(M_t,X_t)\right]
={\mathbb E}_\mathbb{P} \left [ F(x_0^1+W^{1*}_t,x_0+W_t)\right]
\\
&+& \int_0^t {\mathbb E}_\mathbb{P}\left[\partial_m F (x_0^1+W^{1*}_t,x_0+W_t){\mathbf 1}_{\{ W^{1*}_s <W^{1*}_t \}} Z_sB^1(x_0+W_s) \right]ds\nonumber
 \nonumber
\\
&+& \int_0^t {\mathbb E}_\mathbb{P}\left[ \partial_k  F (x_0^1+W^{1*}_t,x_0+W_t)Z_sB^k(x_0+W_s) \right]ds.
\nonumber
\end{eqnarray}
}
As a second step we use the independence of the increments of the Brownian motion in order to make appear the density of $(W^{1*}_{t-s},W_{t-s})$.
{Recall \eqref{X1tu}}: $\hat W_{t-s}:=W_t-W_s$ and $(\hat W^1)^*_{t-s}= \max_{s \leq u \leq t} \left(W^1_u-W^1_s\right).$
Then $W^{1*}_t=
\max \left( W^{1*}_s,W^1_s +
 \max_{s \leq u \leq t} 
 \left(W^1_u-W^1_s\right)\right)
 =\max \left( W^{1*}_s,
 W^1_s +(\hat W^1)_{t-s}^*\right)
$ 
  so the  expression \eqref{34-32} becomes
\begin{align*}
&{\mathbb E}_\mathbb{P} \left [ F(M_t,X_t)\right]
= {\mathbb E}_\mathbb{P} \left [F(x_0^1+W^{1*}_tx_0+,W_t)\right]+
 \\
&  \int_0^t {\mathbb E}_\mathbb{P}
\left[
\partial_m F \left(x_0^1+ W^1_s + 
(\hat W^1)_{t-s}^*
,x_0+W_s+ \hat W_{t-s}\right)
{\mathbf 1}_{\{ W^{1*}_s <W^1_s + (\hat W^1)_{t-s}^*\}}  Z_sB^1(x_0+W_s)
 \right]
ds\nonumber
\\
&+  \int_0^t {\mathbb E}_\mathbb{P}\left[
{\partial_{k}} F
\left(\max \left( x_0^1+W^{1*}_s, 
 x^1_0+W^1_s +
(\hat W^1)_{t-s}^* \right),x_0+W_s+\hat W_{t-s}\right)
 Z_sB^k(x_0+W_s) \right]ds.
\end{align*}
The random vector  $\left(
(\hat W^1)_{t-s}^*
,\hat W_{t-s}\right)$ is independent of the $\sigma-$field $\F_s$ and has the same distribution as the pair $(W^{1*}_{t-s},W_{t-s}).$
Let $p_{W^{*1},W}(.,.;t-s)$ be the density
of its law, { and express the expectation with this density:}
{\footnotesize
\begin{align*}
&{\mathbb E}_\mathbb{P} \left [ F(M_t,X_t)\right]
= \int_{{\mathbb R}^{d+1}}  F(x_0^1+b,x_0+a)p_{W^{*1},W}(b,a;t)dbda\\
&+  \int_0^t {\mathbb E}_\mathbb{P}
\left[\int_{{\mathbb R}^{d+1}}
\partial_m F \left(x_0^1+ W^1_s + b,x_0+W_s+ a\right)
{\mathbf 1}_{\{ W^{1*}_s <W^1_s + b\}}  Z_sB^1(x_0+W_s)p_{W^{*1},W}(b,a;t-s)dbda
 \right]
ds
\\
&+   \int_0^t {\mathbb E}_\mathbb{P}\left[\int_{{\mathbb R}^{d+1}}
{\partial_{k}}  F\left(x_0^1+\max \left( W^{1*}_s, W^1_s + b\right),x_0+W_s+ a\right)
 Z_sB^k(x_0+W_s)p_{W^{*1},W}(b,a;t-s)dbda \right]ds.
 \end{align*}
 }
 Using {Girsanov's} Theorem 
 {for $Z.\mathbb{P}=\mathbb{Q},$ since the law of  $(M,X)$
 under $\mathbb{P}$ is the law of $(x^1_0+W^{1*},x_0+W)$, under $\mathbb{Q}$,}  using
{ the equality   \eqref{girsanov}:} 
{\footnotesize
\begin{align*}
{\mathbb E}_\mathbb{P} \left[F(M_t,X_t)\right]
&= \int_{{\mathbb R}^{d+1}}  F(x_0^1+b,x_0+a)p_{W^{*1},W}(b,a;t)dbda\nonumber
\\
&+  \int_0^t {\mathbb E}_\mathbb{P}
\left[\int_{{\mathbb R}^{d+1}}
\partial_mF \left(X^1_s + b,X_s+ a\right)
{\mathbf 1}_{\{ M_s <X^1_s + b\}}  B^1(X_s)p_{W^{*1},W}(b,a;t-s)dbda
 \right]ds
\nonumber
\\
&+   \int_0^t {\mathbb E}_\mathbb{P}\left[\int_{{\mathbb R}^{d+1}}
{ \partial_{k}}  F\left(\max \left( M_s, X_s^1 + b\right),X_s+ a\right)
 B^k(X_s)p_{W^{*1},W}(b,a;t-s)dbda \right]ds.
\end{align*}
}
\end{proof}
We are now in position to achieve the proof of 
Proposition  \ref{rem-dens-mal}.
Using some  suitable translations of the variables $(a,b),$
 $
{\mathbb E}_\mathbb{P} \left [F(M_t,X_t)\right]
= {\sum_{k=0}^{d} I_k}+I_m.$
where
\begin{align}
\label{rep-dens-4}
&I_0=\int_{{\mathbb R}^{d+1}}  F(b,a)p_{W^{*1},W}(b-x_0^1,a-x_0;t)dbda,
\\
&I_m=\int_0^t {\mathbb E}_\mathbb{P}
\left[\int_{{\mathbb R}^{d+1}}
\partial_m F(b,a)
{\mathbf 1}_{\{ M_s < b\}}  B^1(X_s)p_{W^{*1},W}\left(b-X_s^1,a-X_s;t-s\right)dbda
 \right]
ds\nonumber
\end{align}
and for { $k=1,...,d,$}
\begin{align*}
I_k&= \int_0^t {\mathbb E}_\mathbb{P}\left[\int_{{\mathbb R}^{d+1}}
{{\partial_{k}} }F\left(\max \left( M_s, b\right),a\right)
 B^{k}(X_s)p_{W^{*1},W}(b-X_s^1,a-X_s;t-s)dbda \right]ds.
\end{align*}
 Since $B$, $F$ and its derivatives are bounded,
all these integrals are finite.
{Using  \eqref{dens-cas-mb} 
in Appendix}, the function $p_{W^{*1},W}(.,.;t)$ is $C^{\infty}$ on $\bar{\Delta}=\{(b,a),~~b{\geq} a^1_+, ~~(a,b) \in {\mathbb R}^{d+1}\}$.
\\
{ The aim is now to identify the  terms $p^0,~p^{k,\alpha},~p^{k,\beta},~{\mathbf k=m,1,\cdots, d,}$ defined in Proposition \ref{rem-dens-mal}.} 
\\
{\bf 1.}
Firstly we identify $p^0(b,a;t) $ as the factor of $F(b,a)$ in the integrand of $I_0:$
\begin{equation*}
p^0(b,a;t)=  p_{W^{*1},W}(b-x_0^1,a-x_0;t) .
\end{equation*}
{\bf 2.}
We now  deal with $I_k, k=2,\cdots,d.$ 
  Integrating  by parts
 with respect to { $a^k$ between $-\infty $ and $\infty$ in  $I_k$ for  $k=2,...,d$} yields
\begin{align*}
I_k&=-\int_0^t {\mathbb E}_\mathbb{P}\left[\int_{{\mathbb R}^{d+1}}
 F\left(\max \left( M_s,b\right),a\right)
 B^{k}(X_s)\partial_{k}p_{W^{*1},W}(b-{ X^1_s},a-X_s;t-s)dbda \right]ds\nonumber
 \\
 &=-\int_0^t {\mathbb E}_\mathbb{P}\left[\int_{{\mathbb R}^{d+1}}\1_{\{b> M_s\}}
 F\left(b,a\right)
 B^{k}(X_s)\partial_{k}p_{W^{*1},W}(b-{ X^1_s},a-X_s;t-s)dbda \right]ds
 \\
 &-\int_0^t {\mathbb E}_\mathbb{P}\left[\int_{{\mathbb R}^{d+1}}
 \1_{\{b< M_s\}}F\left( M_s,a\right)
 B^{k}(X_s)\partial_{k}p_{W^{*1},W}(b-{ X^1_s},a-X_s;t-s)dbda \right]ds
 \nonumber
\end{align*}
{ We identify
$-p^{k,\alpha}(b,a,t)$
 inside the integral on the set
  $(b>M_s).$  
  { Concerning the integral on the set
  $(b<M_s)$,} we introduce the density of $(M_s,X_s)$ and identify    $-p^{k,\beta}(m,a;t)$
as factor of  $F(m,a)$.}
\\ 
\\
{  {\bf 3.} Finally, we identify the $p^{m,\gamma}$ and $p^{1,\gamma}$, $\gamma=\alpha, \beta$ which come from the sum of $I_m$ and $I_1$. Note that 
$p_{W^{*1},W}\left(b-X_s^1,a-X_s;t-s\right)=0$
on the set ${\{b<a^1\}}.$
} 
  Integrating  by parts
 with respect to $b$ between $\max\left({ a^1},M_s\right) $ and $\infty$ in $I_m$ yields
\begin{align}
\label{im}
&I_m=-\int_0^t {\mathbb E}_\mathbb{P}\left[\int_{{\mathbb R}^d}F\left( \max\left({ a^1}, M_s\right),a\right)B^1(X_s)p_{W^{*1},W}\left(\max\left({ a^1},M_s\right)-X_s^1,a-X_s;t-s\right)da\right]ds\nonumber
\\
&-\int_0^t {\mathbb E}_\mathbb{P}\left[\int_{{\mathbb R}^{d+1}}{\mathbf 1}_{M_s<b}F\left( b,a\right)B^1(X_s)\partial_{m}p_{W^{*1},W}\left(b-X_s^1,a-X_s;t-s\right)dbda\right]ds
\end{align}
   { Integrating  by parts} with respect to $a^1$ between $-\infty $ and $b$ in $I_1$ yields
\begin{align}
\label{i1}
I_1&= \int_0^t {\mathbb E}_\mathbb{P}\left[\int_{{\mathbb R}^{d}}
F\left( \max\left( M_s,b\right),b,\tilde a\right)
 B^1(X_s)p_{W^{*1},W}(b- X^1_s,b- X^1_s,\tilde a-\tilde X_s;t-s)dbd\tilde a \right]ds
 \nonumber
 \\
&-\int_0^t {\mathbb E}_\mathbb{P}\left[\int_{{\mathbb R}^{d+1}}
 F\left(\max \left( M_s,b\right),a\right)
 B^1(X_s)\partial_{1}p_{W^{*1},W}(b- X^1_s,a-X_s;t-s)dbda \right]ds.
\end{align}
(i) The term  $p^{m,\beta}{(b,a,t)}$ comes from the second term in $I_m$ \eqref{im}
 as the  factor of $F(b,a)$:
$$-\int_0^t {\mathbb E}_\mathbb{P}\left[\int_{{\mathbb R}^{d+1}}{\mathbf 1}_{M_s<b}F\left( b,a\right)B^1(X_s){\partial_m}p_{W^{*1},W}\left(b-X_s^1,a-X_s;t-s\right)dbda\right]ds
$$
(ii) The terms  $-p^{1,\alpha}(b,a,t)$ and  $-p^{1,\beta}(b,a;t)$ come from the second term
 in $I_1$ \eqref{i1}:
$$-\int_0^t {\mathbb E}_\mathbb{P}\left[\int_{{\mathbb R}^{d+1}}
 F\left(\max \left( b,M_s\right),a\right)
 B^1(X_s)\partial_{1}p_{W^{*1},W}(b-{ X^1_s},a-X_s;t-s)dbda \right]ds.
$$
Inside the integral on the set $(M_s<b)$ { we identify $-p^{1,\alpha}(b,a,t)$ and inside
the integral on the set $(M_s>b)$ we identify $-p^{1,\beta}(b,a;t)$
as the factor of $F(b,a)$,} { respectively  
as the factor of $F(M_s,a)$.}
\\
(iii)
The term  $-p^{m,\alpha}(b,a,t)$ comes from the sum of first terms in $I_1$ \eqref{i1}
and $I_m$ \eqref{im}.

 Now  we replace the variable $b$ by $a_1$, $dbd\tilde a$ by $da$ in the first terms of
$I_m$ and $I_1$:
{\footnotesize
\begin{align*}
I^1_m&=-\int_0^t {\mathbb E}_\mathbb{P}\left[\int_{{\mathbb R}^d}F\left( \max\left({ a^1}, M_s\right),a\right)B^1(X_s)p_{W^{*1},W}\left(\max\left({ a^1},M_s\right)-X_s^1,a-X_s;t-s\right)da\right]ds
\\
I^1_1&= \int_0^t {\mathbb E}_\mathbb{P}\left[\int_{{\mathbb R}^{d}}
F\left( \max\left( M_s,a^1\right),a\right)
 B^1(X_s)p_{W^{*1},W}(a^1-{ X^1_s},a-X_s;t-s)d a \right]ds.
\end{align*}
}
Note that
%
\begin{align*}
&-p_{W^{*1},W}\left(\max\left(a^1, M_s\right)-X_s^1,a-X_s;t-s\right)
+
p_{W^{*1},W}(a^1-X^1_s,a-X_s;t-s)
\\
&=\left[-p_{W^{*1},W}\left(M_s-X_s^1,a-X_s;t-s\right)+p_{W^{*1},W}\left(a^1-X_s^1,a-X_s;t-s\right)\right] {\mathbf 1}_{M_s >a^1}
\\
&=-\int_{a^1}^{M_s}
{\partial_m}p_{W^{*1},W}\left(b-X^1_s, a-X_s,t-s\right){ db} {\mathbf 1}_{M_s >a^1}.
\end{align*}
%
Then the sum of  $I_m^1$ and $I_1^1$ is:
\begin{equation*}
-\int_0^t {\mathbb E}_\mathbb{P}\left[\int_{{\mathbb R}^{d+1}}F(M_s,a)B^1(X_s)
{ \partial_m}p_{W^{*1},W}\left(b-X_s^1,a-X_s;t-s\right) {\mathbf 1}_{M_s>b >a^1}dadb\right]ds.
\end{equation*}
We introduce the density of the law of the pair $(M_s,X_s)$ and
 we  identify $-p^{m,\alpha}(m,a;t)$  as
the factor of $F(m,a).$

These three steps  achieve the proof of Proposition \ref{rem-dens-mal} when $\mu_0=\delta_{x_0}.$

Finally when  $\mu_0$ is the law of $X_0,$  we have
$
	p_V(m,w;t,\mu_0)=\int_{{\mathbb R}^d} p_V(m,x;t,\delta_{x_0}) \mu_0(dx_0).
$\\
	Then integrating with respect to $\mu_0$ the expression obtained in { \eqref{rep-p_v}} for $p_V(m,x;t,\delta_{x_0})$
	achieves the proof of Proposition \ref{rem-dens-mal} { for any initial law $\mu_0$}.
\hfill$\Box$

\subsection{ Proof of Proposition \ref{cont-ui--dens}}
\label{regul}

{ Using some idea's used in Garroni section V.3.2 }
let us introduce the linear applications on $L^{\infty}([0, T] ,dt, L^1({\mathbb R}^{d+1},dmdx)),$ $k=m,1,\cdots,d$:
\begin{align}
\label{def-al}
&{ \II^{k,\alpha}}[p](m,x;t):=\int_0^t{\int_{\mathbb{R}^{d+1}} {\mathbf 1}_{b <m}B^{k}(a)\partial_{k}p_{W^{*1},W}\left(m-a^1,x-a;t-s\right)}p(b,a;s)dbdads,
\\
&{ \II^{k,\beta}}[p](m,x;t):=\int_0^t \int_{{\mathbb R}^{d+1}}
{\mathbf 1}_{b<m}
 B^{k}(a)\partial_{k}p_{W^{*1},W}(b-a^1,x-a;t-s)p(m,a;s)
dbda ds.
\nonumber
\end{align}
Let us introduce the functions, 
defined by induction:
\begin{align}
\label{def-suite}
p_0(m,x;t,\mu_0)= \int_{\mathbb{R}^d}p_{W^{1*},W}(m-x_0^1,x-x_0;t)\mu_0(dx_0),~~
p_n=-\sum_{k=m,1,\cdots,d}\left( { p^{k,\alpha}_n} +{ p^{k,\beta}_n}\right)
\end{align}
and for $k=m,1,\cdots,d$,  $j=\alpha,\beta$ and $n \geq 1$, $
{ p_{n+1}^{k,j}}(m,x;t):={\mathcal I}^{k,j}[p_n](m,x;t).$
Let us denote the operator
\begin{equation}
\label{defcalI}
\II:=-\sum_{j=\alpha,\beta;k=m,1,...,d}\II^{k,j}.
\end{equation} 
Moreover { one remarks that this means $p_{n+1}=\II(p_n)$ 
and Proposition \ref{rem-dens-mal} leads to}
$p_V=p_0+\II(p_V).$
Let
\begin{align}
\label{def-gamma_n}
{ P_n} := \sum_{k=0}^n p_k,~~n\geq 0
\end{align}
 \begin{pro}
 \label{pro-dvpt-serie-L1}
{ Assume the vector $B$ is bounded}, then for all $T$ 
the sequence $(P_n)_n$  converges in 
  $L^{\infty} ( [0,T],L^1 ({\mathbb R}^{d+1},dxdm))$ to $p_V$.
Moreover $p_V=\sum_{n=0}^{\infty} p_n.$
 \end{pro}
 The proof  is a consequence of { the two following lemmas.}
\begin{lem}
\label{lem-cont-al}
Let  $j=\alpha,\beta$, $k=m,1,\cdots,d$ and $T>0$
the linear applications ${\mathcal I}^{k,j}$ are continuous on $L^{\infty}([0, T] ,dt,  L^1({\mathbb R}^{d+1},dmdx)):$  there exists a constant $C$ such that for all $p \in L^{\infty}([0, T] ,dt,  L^1({\mathbb R}^{d+1},dmdx)):$ 
\begin{align}
\label{eq:maj-i-racine}
\sup_{s \in [0,t]} \|{\mathcal I}^{k,j}[p](.,.;s)\|_{L^1({\mathbb R}^{d+1},dmdx)}\leq C \int_0^t \frac{1}{\sqrt{t-s}}
\sup_{u \in [0,s]} \|p(.,.;u)\|_{L^1({\mathbb R}^{d+1},dmdx)}ds
\end{align}
{As a consequence,}
\begin{equation}
\label{majII}
\sup_{s \in [0,t]} \|\II[p](.,.;s)\|_{L^1({\mathbb R}^{d+1},dmdx)}\leq 2(d+1)C \int_0^t \frac{1}{\sqrt{t-s}}
\sup_{u \in [0,s]} \|p(.,.;u)\|_{L^1({\mathbb R}^{d+1},dmdx)}ds.
\end{equation}
\end{lem}
\proof 
 Let $T>0,$  $p \in L^{\infty}([0, T] \times L^1({\mathbb R}^{d+1},dmdx))$ and $t \in [0,T]$
    and let $\phi_{d+1}$ be  the Gaussian law density restrained to the subset $\{b>a^1_+\}$
    (up to a constant)
    \begin{align}
  \label{def-phi}
  \phi_{d+1}(b,b-a^1,\tilde a;2t) :=\frac{1}{ {\sqrt{2 \pi t}^{d+1}}}{\mathbf 1}_{b>a^1_+}e^{- \frac{ b^2 + (b-a^1)^2 + \|\tilde{a}\|^2}{4t}}.
  \end{align}
(i) Let $j=\alpha$ and $k=m,1,\cdots,d,$
 according to the definition of ${\mathcal I}^{k,\alpha}$ and the boundedness of $B$,
  \begin{align*}
  \left| {\mathcal I}^{k,\alpha}[p](m,x;t)\right| \leq \|B\|_{\infty} 
 \int_0^t \int_{{\mathbb R}^{d+1}}{\mathbf 1}_{b <m}| \partial_{k}p_{W^{*1},W}\left(m-a^1,x-a;t-s\right)
 p(b,a;s)|dbdads.
  \end{align*}
%
  \\
 {Using Lemma \ref{lem-appendice-1} there 
  exists a constant $D$ such that for  $k=m,1,\cdots,d$:
  \begin{align}
  \label{maj-derk-p0}
  |\partial_{k}p_{W^{*1},W}\left(b,a;t\right)| \leq \frac{D}{\sqrt t}\phi_{d+1}(b,b-a^1,\tilde a;2t).
  \end{align}
 So
    \begin{align*}
  \left| {\mathcal I}^{k,\alpha}[p](m,x;t)\right| \leq \|B\|_{\infty} 
 \int_0^t \int_{{\mathbb R}^{d+1}} \frac{D}{\sqrt{t-s}}\phi_{d+1}(m-a^1,m-x^1,\tilde x-\tilde a;t-s)|p(b,a;s)|dbdads.
  \end{align*}
We operate an integration with respect to $(m,x)$ using Tonelli's theorem and omitting the
 { indicator functions}. Since $\phi_{d+1}$ is the density of a Gaussian law, we get the following bound, 
    \begin{align*}
  \left\| {\mathcal I}^{k,\alpha}[p](.,.;t)\right\|_{L^1({\mathbb R}^{d+1},dmdx)} \leq & D\|B\|_{\infty}
 \int_0^t \int_{{\mathbb R}^{d+1}}\frac{1}{\sqrt{t-s}}| p(b,a;s)|{db}dads
 \\
 \leq & 2^{(d+1)/2}D\|B\|_{\infty}
 \int_0^t \frac{1}{\sqrt{t-s}}\sup_{u\leq s}\| p(.,.;u)\|_{L^1(\mathbb{R}^{d+1},dbda)}ds,
  \end{align*}
meaning inequality \eqref{eq:maj-i-racine} when $j=\alpha.$
\\
\\
(ii) 
Let $j=\beta$ and $k=m,1,\cdots,d.$
According to the definition of ${\mathcal I}^{k,\beta}$ and the boundedness of $B$,
  \begin{align*}
  \left| {\mathcal I}^{k,\beta}[p](m,x;t)\right| \leq \|B\|_{\infty} 
 \int_0^t \int_{{\mathbb R}^{d+1}}{\mathbf 1}_{b <m}| \partial_{k}p_{W^{*1},W}\left(b-a^1,x-a;t-s\right)p(m,a;s)|dbdads.
  \end{align*}
Using \eqref{maj-derk-p0} yields:
$$
  \left| {\mathcal I}^{k,\beta}[p](m,x;t)\right| \leq \|B\|_{\infty} 
 \int_0^t \int_{{\mathbb R}^{d+1}}\frac{D}{\sqrt{t-s}}\phi_{d+1}(b-a^1,b-x^1,\tilde x-\tilde a;2(t-s))|p(m,a;s)|dbdads.
$$
We operate an integration with respect to $x$ then to $b$ using Tonelli's theorem and omitting the
 { indicator functions} and using that
 $\phi$ is the density of a Gaussian law. So the bound with respect to a multiplicative constant:
    \begin{align*}
  \left\| {\mathcal I}^{k,\beta}[p](.,.;t)\right\|_{L^1({\mathbb R}^{d+1},dmdx)} \leq & D \|B\|_{\infty}2^{(d+1)/2}
 \int_0^t \int_{{\mathbb R}^{d+1}}\frac{1}{\sqrt{t-s}}| p(m,a;s)|dmdads\\
 \leq &  D \|B\|_{\infty}2^{(d+1)/2}
 \int_0^t \frac{1}{\sqrt{t-s}}\sup_{u\leq s}\| p(.,.;u)\|_{L^1(\mathbb{R}^{d+1},dmda)}ds,
  \end{align*}
meaning inequality  \eqref{eq:maj-i-racine} for $j=\beta.$

{ Finally, estimation \eqref{majII} is obtained by adding estimations 
 \eqref{eq:maj-i-racine} for $j=\alpha, \beta$ and $k=m, 1,...,d.$}
\hfill$\bullet$
\\


%
%
The following lemma is a consequence of \eqref{majII} in Lemma \ref{lem-cont-al}:
\begin{lem}
\label{lem-maj-unif-p-nl1}
For all $n $
\begin{align}
 \label{maj-norme-pn-l1}
\sup_{u \leq t} \|p_n(.,.;u)\|_{L^1(\mathbb{R}^{d+1},dmdx)} \leq (2(d+1)C)^{n} t^{n/2} \frac{\Gamma(1/2)^n}{\Gamma( 1+n/2)},
\end{align}
\begin{align} 
\label{maj-norme-gamma-l1}
\sup_{u \leq t }\|(p_V -P_{n})(.,.;u)\|_{L^1(\mathbb{R}^{d+1},dmdx)} \leq (2(d+1)C)^{n+1} t^{(n+1)/2} \frac{\Gamma(1/2)^{n+1}}{{\Gamma( (n+3)/2)}}.
\end{align}
\end{lem}
\proof
(i) For all $t>0,$ $p_0(.;t)$ is a density of probability, so   \eqref{maj-norme-pn-l1} is satisfied for $n=0.$ We now assume that  \eqref{maj-norme-pn-l1} 
is satisfied for $n.$ Using $p_{n+1}=\II [p_n]$,
  \eqref{majII} and the induction e assumption:
 \begin{align*}
 \sup_{u \leq t} \|p_{n+1}(.,.;u)\|_{L^1(\mathbb{R}^{d+1},dmdx)} \leq (2(d+1) C)^{n+1}  \frac{\Gamma(1/2)^n}{\Gamma( 1+n/2)}\int_0^t \frac{\sqrt{s^n}}{\sqrt{t-s}}ds.
 \end{align*}
We operate the change of variable $s=tu$ and use $\int_0^1 u^{a-1}(1-u)^{b-1}du= \frac{\Gamma(a)\Gamma(b)}{\Gamma(a+b)}$:
  \begin{align*}
 \sup_{u \leq t }\|p_{n+1}(.,.;u)\|_{L^1(\mathbb{R}^{d+1},dmdx)} \leq (2(d+1) C)^{n+1} t^{(n+1)/2} \frac{\Gamma(1/2)^n}{\Gamma( 1+n/2)} \frac{\Gamma(1/2) \Gamma( 1+ n/2)}{\Gamma((n+3)/2)}
 \end{align*}
which proves \eqref{maj-norme-pn-l1} for all  $n.$


(ii) Noting that
$P_0=p_0$ and
$p_V-p_0= \II[p_V]$ and
{  applying  \eqref{majII} to  $p_V$ yield
 \begin{align*}
 \sup_{u \leq t }\|(p_V -P_{0})(.,.;u)\|_{L^1(\mathbb{R}^{d+1},dmdx)} \leq 2(d+1)C t^{1/2}.
 \end{align*}
But $\Gamma(1/2) / \Gamma(3/2)=2$ so \eqref{maj-norme-gamma-l1}  is satisfied for $n=0.$
}

We now suppose that \eqref{maj-norme-gamma-l1} is satisfied for $n.$  
Using $p_V-P_{n+1}=$\\$
p_0+\II(p_V)- (p_0+\II(P_n))=\II(p_V-P_n),
$ 
the bound \eqref{majII} and the induction assumption:
 \begin{align*}
 \sup_{u \leq t }\|[p_{V}- P_{n+1}](.,.;u)\|_{L^1(\mathbb{R}^{d+1},dmdx)} \leq 2(d+1)C
 \int_0^t 
 (2(d+1) C)^{n+1}  \frac{\Gamma(1/2)^{n+1}}{\Gamma( (3+n)/2)} \frac{\sqrt{s^{n+1}}}{\sqrt{t-s}}ds.
 \end{align*}
We now operate the change of  variable $s=tu$ and  $\int_0^1 u^{a-1}(1-u)^{b-1}du= \frac{\Gamma(a)\Gamma(b)}{\Gamma(a+b)}$ { with $a=(n+3)/2,b=\demi$:
  \begin{align*}
 \sup_{u \leq t }\|[p_V-P_{n+1}](.,.;u)\|_{L^1(\mathbb{R}^{d+1},dmdx)} \leq (2(d+1) C)^{n+2} t^{(n+2)/2} \frac{\Gamma(1/2)^{n+2}}{{\Gamma( (4+n)/2)}} 
 \end{align*}
which proves  \eqref{maj-norme-gamma-l1} for $n+1$ and thus for all $n.$
\hfill$\Box$

The series  $\sum_n \frac{x^n }{\Gamma(n/2 +1)}$ is convergent for any $x$, so {
 Proposition \ref{pro-dvpt-serie-L1} is 
 a consequence of lemmas 
 \ref {lem-cont-al} and \ref{lem-maj-unif-p-nl1}}.}
 %

 \subsubsection{Upper Bound of $p_V$}
 
 meaning  Hypothesis 2.1 (i).
 \\
For all $T>0,$ $x_0 \in \mathbb{R}^d,$ $p \in L^{\infty}([0,T], L^{1}(\mathbb{R}^{d+1},dmdx))$ the support of which being included in
$\{(m,x),~~m>x_0^1, m>x^1\}$ let us denote
  \begin{align} 
  \label{maj-p-gaus}
N(p;t,x_0) :=\sup_{(m,x)\in \mathbb{R}^{d+1},~~m>x^1, m>x_0^1}\frac{ |p(m,x;t)|}{\phi_{d+1}(m-x_0^1, m-x^1,\tilde{x}-\tilde{x}_0;2t)}.
 \end{align}

 \begin{pro}
 \label{pro-maj-pv-g} 
For all $T>0$ there exists a constant $C_T$ 
 and for all $n$ there exists constants 
 $C_n= \frac{\left[\|B\|_\infty D{(2(d+1)) 2^{d/2} }{\Gamma(1/2)}\right]^n }{\Gamma(1+n/2)} $
 such that: for all  $x_0 \in \mathbb{R}^d$,  ${0<t\leq T},$
 \begin{align*}
& (i)&\left| p_n(m,x;t,x_0) \right| \leq C_nt^{n/2}\phi_{d+1}( m-x^1_0,m-x^1,\tilde{x}-\tilde{x}_0, 2t){\mathbf 1}_{m>\max(x^1,x^1_0)}
 \end{align*}

(ii)
{
$\left|p_V(m,x;t,x_0) \right| \leq C_T\phi_{d+1}( m-x^1_0,m-x^1,\tilde{x}-\tilde{x}_0, 2t){\mathbf 1}_{m>\max(x^1,x^1_0)}$}

{ (iii) For all $\mu_0$ initial probability measure on ${\mathbb R}^d$, 
\\
$\sup_{u >0} p_V(m,m-u,\tilde{x},t; \mu_0) \in L^1([0,T] \times \mathbb{R}^d, dtdmd\tilde{x}).$}
 \end{pro}
  { Remark that, actually,  this point (iii) is Hypothesis \ref{hyp-dens-outil} (i). }

 \proof
Point $(ii)$ is a consequence 
 of point (i), since $p_V=\sum_{n=0}^{\infty}p_n,$  and the series 
 $\sum_n \frac{1}{{\Gamma (1+n/2)}}x^n$ admits an  infinite  radius of convergence (Proposition \ref{pro-dvpt-serie-L1}). 
 \\
We prove point $(i)$ by induction on $n$   using point (ii) in  Lemma \ref{lem-appendice-1}: $p_0(m,x;t,x_0) \leq $
\begin{align*}
\frac{e^{-\frac{(m-x^1)^2}{4t}-\frac{\|\tilde{x}-\tilde{x}_0\|^2}{4t} {- \frac{(m-x_0^1)^2}{4t} }}}{\sqrt{(2 \pi )^{d+1}{t^{d+1}}}}{\mathbf 1}_{m>\max(x^1,x^1_0)}=
 \phi_{d+1}(m-x^1,m-x_0^1,\tilde x-\tilde x_0;2t){\mathbf 1}_{m>\max(x^1,x^1_0)},
\end{align*}
so $N(p_0;t,x_0) \leq 1,$ which is (i) for $n=0,$ {$C_0=1$.}

We assume point $(i)$ is true for $p_n,$ meaning $N(p_n;t,x_0) \leq  C_nt^{n/2}.$
By definition $p_{n+1}=  {\mathcal I}[p_n],$   Lemma \ref{maj-ikj-x1m} proved below yields:
\begin{align*}
N(p_{n+1};t,x_0)=N(\II[p_n];t,x_0)&\leq  2(d+1)2^{d/2}\|B\|_{\infty}DC_n \int_0^t \frac{s^{n/2}}{\sqrt{2\pi (t-s)} }ds.
\end{align*}
We operate the change of variable $s=tu$ 
\begin{align*}
N(p_{n+1};t,x_0)\leq \frac{2(d+1)2^{d/2}\|B\|_{\infty}D}{\sqrt{2\pi}}C_n {(\sqrt{t})^{n+1}}\int_0^1 \frac{u^{n/2}}{\sqrt{1-u}} ds
\end{align*}
Using  $\int_0^1 \frac{u^{n/2}}{\sqrt{ 1-u}} du= \frac{\Gamma((n+2)/2 )\Gamma(1/2)}{\Gamma((n+3)/2)}$ and $C_n$ definition:
\begin{align*}
N(p_{n+1};t,x_0)\leq C_{n+1} (\sqrt{t})^{n+1},
\end{align*}
this achieves the proof of point (i) in Proposition \ref{pro-maj-pv-g}.
\\
{ (iii) Then for all $x_0 \in \mathbb{R}^d$ and using $x^1=m-u,$
$$\sup_{u>0}p_V(m,m-u,\tilde x;t)\leq
 C_T\phi_{d+1}(m-x^1_0,0,\tilde{x}-\tilde{x}_0;2t)
 \in L^1([0,T]\times \mathbb{R}^d,dtdmd\tilde x).$$
Since $p_V(m,x;t,\mu_0)=\int_{\mathbb{R}^d} p_V(m,x;t,x_0) \mu_0(dx_0)$ 
 point (iii) is true.}
\hfill$\Box$

 \begin{lem}
 \label{maj-ikj-x1m}
Let $T>0, $ $x_0 \in \mathbb{R}^d,$ $p \in L^{\infty}([0,T], dt, L^{1}(\mathbb{R}^{d+1},dmdx))$ such that the support of $p(.,.;t)$ is included in $\{(m,x),~~m>x_0^1, m>x^1\}$ and for all $s\in ]0,T]$ $N(p;s,x_0) <\infty. $ Then for $j=\alpha,$ $k=m,1,\ldots,d,$ the support of function ${\mathcal I}^{j,k}[p](.;t)$ is included in $\{(m,x),~~m>x_0^1, m>x^1\}.$ Moreover for all $t\in [0,T]$ we have :
 \begin{align*}
N({\mathcal I}[p];t,x_0)\leq 2(d+1)2^{d/2}\|B\|_{\infty}D \int_0^t \frac{1}{\sqrt{2\pi { (t-s)}}}N(p;s,x_0)ds,~~\forall t \in[0,T].
\end{align*}
\end{lem}
\proof Let $T>0,$ $x_0\in\mathbb{R}^d,$ $p \in L^{\infty}([0,T], dt,L^{1}(\mathbb{R}^{d+1},dmdx))$
such that for all $t>0$ the support of $p(.;t)$ is included in $\{(m,x),~~m>x_0^1, m>x^1\}.$

(i) For $j=\alpha,$ $k=m, 1,\cdots,d,$ 
using the definition of ${\mathcal I}^{\alpha,k}$  
 yields:
 \begin{align*}
 { {\mathcal I}^{k,\alpha}}[p](m,x;t):=\int_0^t{\int_{\mathbb{R}^{d+1}} B^{k}(a)\partial_{k}p_{W^{*1},W}\left(m-a^1,x-a;t-s\right){\mathbf 1}_{x_0^1<b <m,m>x^1 }p(b,a;s)dbda}ds
 \end{align*}

So the support of ${\mathcal I}^{\alpha,k}[p](.;t)$ is included in $\{(m,x) \in \mathbb{R}^{d+1},~~m>\max(x^1_0,x^1)\}.$  For now on, we only consider $(m,x)$ such that $m >\max(x^1,x^1_0).$
 
{ Let $p$  be a function such that} 
 $\forall s \in ]0,T]$ $N(p;x_0,s)<\infty.$
The definition of
 ${\mathcal I}^{k,\alpha},$ the boundedness of $B,$ the fact that $\partial_kp_{W^*,W}$
satisfies  \eqref{maj-derk-p0} and the definition \eqref{maj-p-gaus} of $N(p;t,x_0)$  imply
 \begin{align*}
 \left| {\mathcal I}^{k,\alpha}[p](m,x;t) \right| \leq \|B\|_{\infty}\int_0^t\int_{{\mathbb R}^{d+1}}N(p;s,x_0)
 \frac{D}{ \sqrt{ (t-s)}}
 {\mathbf 1}_{m>x_1}{\mathbf 1}_{m>b>\max(a^1,x_0^1)}
 \\
 ~~~~~~ \phi_{d+1}(m-a^1, m-x^1,\tilde{x}-\tilde{a};t-s)
 \phi_{d+1}(b-x_0^1,b-a^1,\tilde{a}-\tilde{x_0};s) db da ds.
 \end{align*}

We integrate in $\tilde{a}$ using Lemme \ref{lem-convol-gaussien} (ii) with $u=\tilde{x},$ $v=\tilde{a},$ $w=\tilde{x}_0$ and the fact that $\phi_{d+1}$ is a Gaussian density of probability:
  \begin{align}
  \label{I-alpha-int-1}
& \left| {\mathcal I}^{\alpha,k}[p](m,x;t) \right| \leq  2^{(d-1)/2}\|B\|_{\infty}D
 \\
& \int_0^t\int_{{\mathbb R}^{2}}N(p;s,x_0)
 {\mathbf 1}_{m>b>\max(a^1,x_0^1)} \frac{e^{-\frac{\|\tilde{x}-\tilde{x_0}\|^2}{4t}}}{\sqrt{(2 \pi t)^{d-1}}}
 \frac{e^{-\frac{(m-a^1)^2}{4(t-s)}-\frac{(m-x^1)^2}{4(t-s)} }}{\sqrt{ (2\pi)^2(t-s)^3}}
 \frac{e^{-\frac{(b-x_0^1)^2}{ 4s}-\frac{(b-a^1)^2}{4s}}}{\sqrt{(2\pi)^2 {s^2}}}db da^1 ds.
 \nonumber
 \end{align}
Using point (i')  Lemma \ref{lem-appendice-1}
 with $u=m,$ $v=a^1,$ $w=b,$ $k=1$, we integrate in $da^1$ up to $b$:
 \begin{align*}
&\int_{-\infty}^b 
 \frac{e^{-\frac{(m-a^1)^2}{4(t-s)} }}{\sqrt{ 2\pi(t-s)}}\frac{e^{-\frac{(b-a^1)^2}{4s}}}{\sqrt{(2\pi s)}}da^1
=
\frac{e^{\frac{-(m-b)^2}{4t}}}{\sqrt{2\pi t}}
\Phi_G\left( \sqrt{ \frac{s}{2t(t-s)}}(b-m)\right)
\end{align*}
where $ \Phi_G(u) =\int_{-\infty}^u e^{-z^2/2}dz\leq \demi e^{-u^2/2}$ for $u=b-m<0$ 
according to Lemma \ref{lem-convol-gaussien} (iii). This  yields  the bound:
$
\frac{e^{\frac{-(m-b)^2}{4t}}}{\sqrt{2\pi t}}
e^{-\frac{s(b-m)^2}{4t(t-s)}}$
and
 \begin{align*}
2\int_{-\infty}^b  \frac{e^{-\frac{(m-a^1)^2}{4(t-s)} }}{\sqrt{ 2\pi(t-s)}}\frac{e^{-\frac{(b-a^1)^2}{4s}}}{\sqrt{(2\pi s)}}da^1
&\leq
\frac{e^{-\frac{(m-b)^2}{4t}}}{\sqrt{2\pi t}}
e^{-\frac{s(m-b)^2}{4t(t-s)} }
= \frac{e^{-\frac{(m-b)^2}{4(t-s)}}}{\sqrt{2\pi  t}}.
\end{align*}
Plugging this inequality inside \eqref{I-alpha-int-1} yields with $C_{d,B}=2^{(d+1)/2}\|B\|_{\infty}D$
  \begin{align*}
 \frac{\left| {\mathcal I}^{\alpha,k}[p](m,x;t)\right|}{C_{d,B}}  \leq 
\int_0^t\int_{{\mathbb R}}N(p;s,x_0)
 {\mathbf 1}_{m>b>x_0^1} \frac{e^{-\frac{\|\tilde{x}-\tilde{x_0}\|^2}{4t}}}{\sqrt{(2 \pi t)^d}}
 \frac{e^{-\frac{(m-b)^2}{4(t-s)}-\frac{(m-x^1)^2}{4(t-s)} }}{\sqrt{ 2\pi(t-s)^2{s}}}
 e^{-\frac{(b-x_0^1)^2}{ 4s}}db  ds.
 \end{align*}
Omitting the  indicator functions, Lemma \ref{lem-convol-gaussien} (ii) with $u=m,~v=b,~w=x_0^1,~k=1$
implies
$$
\int_{b<m}\frac{e^{-\frac{(m-b)^2}{4(t-s)}}e^{-\frac{(b-x_0^1)^2}{ 4s}}}{\sqrt{ 2\pi(t-s)2\pi s}}
 db \leq
\sqrt{\frac{2}{2\pi t}}e^{-\frac{(m-x_0^1)^2}{4t}}.
$$
Inserting this result,  we obtain
   \begin{align*}
 \left| {\mathcal I}^{\alpha,k}[p](m,x;t) \right| \leq \sqrt{2}C_{d,B}\int_0^tN(p;s,x_0)
  \frac{e^{-\frac{\|\tilde{x}-\tilde{x_0}\|^2}{4t}}}{\sqrt{(2 \pi t)^{d+1}}}
 \frac{e^{-\frac{(m-x_0^1)^2}{4t}-\frac{(m-x^1)^2}{4(t-s)} }}{\sqrt{ 2\pi{(t-s)}}}
  ds.
 \end{align*}
For $0<s<t,$ $e^{-\frac{(m-x^1)^2}{4(t-s)}}\leq e^{ -\frac{(m-x^1)^2}{4t}},$ so
    \begin{align*}
 \left| {\mathcal I}^{\alpha,k}[p](m,x;t) \right| \leq \sqrt{2} C_{d,B}\int_0^tN(p;s,x_0)
  \frac{e^{-\frac{\|\tilde{x}-\tilde{x_0}\|^2}{4t}}}{\sqrt{(2 \pi t)^{d+1}}}
 \frac{e^{-\frac{(m-x_0^1)^2}{4t}-\frac{(m-x^1)^2}{4t} }}{\sqrt{ 2\pi{(t-s)}}}
  ds.
 \end{align*}
 Using the definition of $\phi_{d+1}$ we identify 
     \begin{align*}
 \left| {\mathcal I}^{\alpha,k}[p](m,x;t) \right| \leq \sqrt{2} C_{d,B}\int_0^tN(p;s,x_0)
 \frac{\phi_{d+1}(m-x_0^1,m-x^1,\tilde{x}-\tilde{x_0};2t)}{\sqrt{ 2\pi {(t-s)}}}
  ds
 \end{align*}
and with the definition of $N,$  with respect to a multiplicative  constant:
     \begin{align*}
N( {\mathcal I}^{\alpha,k}[p],x_0,t)
 \leq \sqrt{2} C_{d,B}
 \int_0^t N(p;s,x_0)
 \frac{1}{\sqrt{ 2\pi{(t-s)}}}
  ds.
 \end{align*}
(ii) For  $j=\beta,$ $k=m,1,\cdots,d$
using the definition of ${\mathcal I}^{\beta,k}$ and the fact that the support of $p$ is
included in $\{(m,x),~~m>\max(x_0^1,x^1)\}$ yields $ {\mathcal I}^{\beta,k}[p](m,x;t)= $
 \begin{align*}
\int_0^t \int_{\mathbb{R}^{d+1}}{\mathbf 1}_{m>b>x^1, m>x_0^1,b>a^1} B^{k}(a) \partial_k p_{W^{1*},W}(b-a^1,x-a,t-s)p(m,a,s) da db ds.
 \end{align*}
Thus the support of  $ {\mathcal I}^{\beta,k}[p](.;t)$ is included in $\{(m,x),~~m>\max(x^1,x^0)\}.$
For now on we only consider $(m,x)$ satisfying $m>\max(x^1,x_0^1).$  Definition of ${\mathcal I}^{\beta,k},$ the boundedness of $B,$ 
 the inequality  \eqref{maj-derk-p0}  satisfied by $\partial_kp_{W^*,W}$:
\\
$|\partial_k p_{W^{1*},W}(b-a^1,x-a,t-s)|
\leq \frac{D}{\sqrt{t-s}}
\phi_{d+1}(b-a^1,b-x^1,\tilde x-\tilde a,2(t-s)) $
\\
and the definition of $N(p;t,x_0)$ yield:
 \begin{align*}
 \left| {\mathcal I}^{\beta,k}[p](m,x;t)\right|\leq  \|B\|_{\infty} D \int_0^t \int_{\mathbb{R}^{d+1}}{\mathbf 1}_{m>b>x^1, b>a^1} N(p;s,x_0)\\
\frac{e^{- \frac{(b-a^1)^2}{4(t-s)}- \frac{(b-x^1)^2}{4(t-s)} - \frac{\|\tilde{x}-\tilde{a}\|^2}{4(t-s)}}}{\sqrt{(2\pi)^{d+1}(t-s)^{d+2}}}\frac{e^{- \frac{(m-x_0^1)^2}{4s}- \frac{(m-a^1)^2}{4s} -\frac{\|\tilde{x}_0 -\tilde{a}\|^2}{4s}}}{ \sqrt{(2\pi)^{d+1}{s^{d+1}}}}da db ds.
 \end{align*}
We integrate in $\tilde{a} $ using  Lemma \ref{lem-convol-gaussien} (ii) with $u=\tilde{x},$ $v=\tilde{a}$ et $w=\tilde{x}_0$:
  \begin{align}
  \label{eq:-beta-rac}
& \left| {\mathcal I}^{\beta,k}[p](m,x;t)\right|\leq C_{d,B}.
 \\
 & \int_0^t \int_{\mathbb{R}^{2}}{\mathbf 1}_{m>b>x^1, b>a^1}\frac{e^{-\frac{\|\tilde{x}-\tilde{x}_0\|^2}{4t} }}{\sqrt{(2 \pi t)^{d-1}}}N(p;s,x_0)\nonumber
.\frac{e^{- \frac{(b-a^1)^2}{4(t-s)}- \frac{(b-x^1)^2}{4(t-s)} }}{\sqrt{(2\pi)^2(t-s)^{3}}}\frac{e^{- \frac{(m-x_0^1)^2}{4s}- \frac{(m-a^1)^2}{4s} }}{ \sqrt{(2\pi)^{2}{s^{2}}}}da^1 db ds.
 \end{align}
Using  Lemma \ref{lem-convol-gaussien} (i') for $u=b,$ $v=a^1,$ $w=m$ $k=1$
  \begin{align*}
 & \int_{-\infty}^b \frac{e^{- \frac{(b-a^1)^2}{4(t-s)} }}{\sqrt{2\pi(t-s)}}\frac{e^{- \frac{(m-a^1)^2}{4s} }}{ \sqrt{2\pi s}}da^1= \sqrt{2}\frac{e^{- \frac{(b-m)^2}{4t}}}{\sqrt{2 \pi t}}\Phi_G\left( \sqrt{\frac{t}{4s(t-s)}}[ b - (\frac{s}{t} b + \frac{t-s}{t} m)] \right)\\
 &= \frac{e^{- \frac{(b-m)^2}{4t}}}{\sqrt{2 \pi t}}\Phi_G\left( \sqrt{\frac{t-s}{4st}}[ b - m] \right)\leq  \frac{e^{- \frac{(b-m)^2}{4t}}}{\sqrt{2 \pi t}}e^{- \frac{t-s}{4st}[ b - m]^2}= \frac{e^{- \frac{(b-m)^2}{4s}}}{\sqrt{2 \pi t}}
  \end{align*}
the last bound coming from  Lemma \ref{lem-convol-gaussien} (iii) since $b-m<0.$
\\
We plugg this estimation in \eqref{eq:-beta-rac}

  \begin{align*}
& \left| {\mathcal I}^{\beta,k}[p](m,x;t)\right|\leq 
 \\
& C_{d,B}\int_0^t \int_{\mathbb{R}}{\mathbf 1}_{m>b>x^1}\frac{e^{-\frac{\|\tilde{x}-\tilde{x}_0\|^2}{4t} }}{\sqrt{(2 \pi t)^{d}}}N(p;s,x_0)
\frac{ e^{-\frac{(b-x^1)^2}{4(t-s)} }}{\sqrt{2\pi(t-s)^{2 }}}  \frac{e^{- \frac{(m-x_0^1)^2}{4s}-\frac{(b-m)^2}{4s} }}{ \sqrt{2\pi s}} db ds.
 \end{align*}

We integrate with respect to  $b$ on  $\mathbb{R}$ and we
 use Lemma \ref{lem-convol-gaussien} (ii) with $u=x^1,$ $v=b$, $w=m,$
   $k=1$:
      \begin{align*}
{ \left| {\mathcal I}^{\beta,k}[p](m,x;t)\right|\leq \sqrt{2}C_{d,B}
 \int_0^t 
\frac{e^{-\frac{(m-x^1)^2}{4t}-\frac{\|\tilde{x}-\tilde{x}_0\|^2}{4t} }}{\sqrt{(2 \pi t)^{d+1}}}N(p;s,x_0)
 \frac{e^{- \frac{(m-x_0^1)^2}{4s} }}{ \sqrt{2\pi (t-s)}} ds}.
 \end{align*}
 When $0<s<t,$ $e^{- \frac{(m-x_0^1)^2}{4s} }\leq e^{- \frac{(m-x_0^1)^2}{4t} }$ so:
    \begin{align*}
 \left| {\mathcal I}^{\beta,k}[p](m,x;t)\right|\leq \sqrt{2}C_{d,B}\int_0^t \frac{e^{-\frac{(m-x^1)^2}{4t}-\frac{\|\tilde{x}-\tilde{x}_0\|^2}{4t} {- \frac{(m-x_0^1)^2}{4t} }}}{\sqrt{(2 \pi t)^{d+1}}}N(p;s,x_0)
 \frac{1}{ \sqrt{2\pi (t-s)}}  ds.
 \end{align*}
{  Under the integral we identify the factor}  $\phi_{d+1}(m-x^1_0,m-x^1, \tilde{x}-\tilde{x_0};2t)$ so
     \begin{align*}
 \left| {\mathcal I}^{\beta,k}[p](m,x;t)\right|\leq  \sqrt{2}C_{d,B}\phi_{d+1}(m-x^1_0,m-x^1, \tilde{x}-\tilde{x_0};2t)\int_0^t N(p;s,x_0)
 \frac{1}{ \sqrt{2\pi {(t-s)}}} db ds.
 \end{align*}
Finally using the definition of $N$ \eqref{maj-p-gaus} we  have proved
   \begin{align*}
N( {\mathcal I}^{\beta,k}[p],x_0,t) \leq \sqrt{2} C_{d,B}\int_0^tN(p;s,x_0)
 \frac{1}{\sqrt{ 2\pi{(t-s)}}}
  ds
 \end{align*}
which achieves the  proof of Lemma \ref{maj-ikj-x1m}. 
\hfill $\Box$
 

\subsubsection{Proof of Hypothesis \ref{hyp-dens-outil} (ii), case $A=I_d$}
\begin{pro}
\label{cont-bis}
For any $\mu_0$ probability measure on ${\mathbb R}^d,$
 for all $(m, \tilde{x},t)\in \mathbb{R}^d\times ]0,T],$    $u \mapsto p_V(m,m-u,\tilde{x},t)$ admits a limit when $u$ goes to 0, $u>0.$
\end{pro}
\proof

%
%

{ The proof  is a consequence of the three}  following lemmas.
 \begin{lem}
\label{cont-ui--p0bis}
Recall that { $
p^0(m,x;t)=\int_{{\mathbb R}^d}p_{W^{*1},W}(m-x_0^1,x-x_0;t)\mu_0(dx_0) .
$}
\begin{align*}
 ~~\lim_{u\rightarrow 0,u>0} p^0(b,b-u,\tilde{a};t) =p^0(b,b,\tilde{a};t) ,~~\forall u>0, ~~(b,\tilde{a}) \in {\mathbb R}^d, {\forall t>0}.
\end{align*}
\end{lem}
\begin{proof}
 We have
$
p^0(b,b-u,\tilde{x};t)=
{\int_{{\mathbb R}^d}} { 2\frac{b+u-x_0^1}{\sqrt{(2\pi)^d 
{\mathbf{ t^{d+1}}}}}e^{-\frac{(b+u-x_0^1)^2}{2t} -\frac{\|\tilde{x}-\tilde{x}_0\|^2}{2t}}}{ {\mathbf 1}_{b \geq x_0^1,~~u \geq 0}\mu_0(dx_0)}.
$
Then, { since the integrand is {dominated}  by
{ $\frac{{ D}}{\sqrt{(2 \pi)^d t^{d+1}}}$}
 and $\mu_0$ is a probability measure, using {Lebesgue's dominated convergence} Theorem} yields:
\begin{align*}
\lim_{u \rightarrow 0,u>0} p^0(b,b-u,\tilde{x};t)=p^0(b,b,\tilde{x};t),~~\forall ~~(b,\tilde{x}) \in {\mathbb R}^d, { \forall t>0}.
\end{align*}
\end{proof}
\begin{lem}
\label{lem-cont-alpha-bis}
For $k=m,1,...,d$ recall that
$$p^{k,\alpha}(m,x;t)=\int_0^t\int_{\mathbb{R}^{d+1}} {\mathbf 1}_{b<m } B^k(a) \partial_{k}p_{W^{*1}, W} (m-a^1,x-a, t-s) p_V(b,a;s) db da ds.
$$
The map $u \mapsto p^{k,\alpha}(m,m-u,\tilde{x};t)$ converges to $p^{k,\alpha}(m,m,\tilde{x};t)$ when $u$ goes to $0^+.$
\end{lem}

\begin{proof}
The proof wil be a consequence of Lebesgue dominated theorem.
First, the map $u \mapsto {\mathbf 1}_{b<m } \partial_{k}p_{W^{*1}, W} (m-a^1,m-u-a^1,\tilde{x}-\tilde{a}; t-s) p_V(b,a;s)$ {converges to ${\mathbf 1}_{b<m } \partial_{k}p_{W^{*1}, W} (m-a^1,ma^1,\tilde{x}-\tilde{a}; t-s) p_V(b,a;s)$  when $u$ goes to $0^+.$}

Second it is dominated by $q^{k,\alpha}(m,\tilde{x},a,b;s,x_0):=$
\begin{align*}
|B^k(a)|{\mathbf 1}_{b<m }\sup_{u>0}  \left|\partial_{k}p_{W^{*1}, W} (m-a^1,m-u-a^1,\tilde{x}-\tilde{a}; t-s) p_V(b,a;s,x_0)\right|
\end{align*}

We seek to prove that

\begin{align}
\label{maj-lebesgue-p-aplpha}
\int_0^T \int_{\mathbb{R}^{2d+1}}q^{k,\alpha}(m,\tilde{x},a,b;s,x_0) ds db da \mu_0(dx_0)<+\infty.
\end{align}
According to estimation \eqref{maj-derk-p0} of $\partial_k p_{W^{*1},W}$ and estimation (ii) of Proposition \ref{pro-maj-pv-g},  we obtain
\begin{align*}
q^{k,\alpha}(m,\tilde{x},a,b;s,x_0)&\leq 
\|B\|_{\infty}{\mathbf 1}_{ m>b>a^1}\frac{ D}{\sqrt{t-s} \sqrt{2\pi(t-s)}^{d+1} }
 \exp [ -\frac{(m-a^1)^2}{4(t-s)} -\frac{\|\tilde{x}-\tilde{a}\|^2}{4(t-s)} ]
 \\
 &\frac{C_T}{\sqrt{2\pi s}^{d+1} }\exp[- \frac{(b-x_0^1)^2}{4s} - \frac{(b-a^1)}{4s} - \frac{\|\tilde{x}_0-\tilde{a}\|^2}{4s}].
\end{align*}
We integrate with respect to $\tilde{a}$ using Lemma  \ref{lem-convol-gaussien} (ii) for $k=d+1,$ $u=\tilde{x},$ $v=\tilde{a}$ and $w=\tilde{x}_0:$
\begin{align*}
\int_{\mathbb{R}^{d-1}}q^{k,\alpha}(m,\tilde{x},a^1,b;s,x_0^1)d\tilde{a} &\leq {\mathbf 1}_{ m>b>a^1}\frac{\|B\|_{\infty} C_T D2^{(d-1)/2}}{\sqrt{t-s} \sqrt{2\pi(t-s)}^{2} \sqrt{2\pi s}^{2}}
\frac{e^{-\frac{|\tilde{x}-\tilde{x}_0\|^2}{4t}}}{\sqrt{2 \pi t}^{d-1}}\\
~~~~~~~~~~~~~~~~~~~~~~~~~~~& \exp [ -\frac{(m-a^1)^2}{4(t-s)} - \frac{(b-x_0^1)^2}{4s} - \frac{(b-a^1)^2}{4s} ].
\end{align*}
We integrate with respect to $a_1$ between $-\infty$ and $b$  using Lemma \ref{lem-convol-gaussien} {\bf (i')}
for  $u=m,$ $v=a^1$ and $w=b$
\begin{align*}
\int_{\mathbb{R}^{d}}{\mathbf 1}_{ a^1<b}q^{k,\alpha}(m,\tilde{x},b,a;s;x_0)da &\leq {\mathbf 1}_{ b<m}\frac{\|B\|_{\infty} C_T D2^{d/2}}{\sqrt{t-s} \sqrt{2\pi(t-s)} \sqrt{2\pi s}}
\frac{e^{-\frac{|\tilde{x}-\tilde{x}_0\|^2}{4t}}-\frac{(b-m)^2}{4t}}{\sqrt{2 \pi t}^{d}}
\\
& \exp [ - \frac{(b-x_0^1)^2}{4s} ]\Phi_G\left(\sqrt{\frac{t}{2s(t-s)}}( b- [\frac{s}{t}m + \frac{(t-s)}{t}b])\right) .
\end{align*}
Note that $ \sqrt{\frac{t}{2s(t-s)}}( b- [\frac{s}{t}m + \frac{(t-s)}{t}b])= \sqrt{\frac{s}{2t(t-s)}}(b-m) $ and using Lemma \ref{lem-convol-gaussien} (iii)
\begin{align*}
\int_{\mathbb{R}^{d}}q^{k,\alpha}(m,\tilde{x},b,a;s,x_0)da &\leq {\mathbf 1}_{ b<m}\frac{\|B\|_{\infty} C_T D2^{d/2}}{\sqrt{t-s} \sqrt{2\pi(t-s)} \sqrt{2\pi s}}
\frac{e^{-\frac{|\tilde{x}-\tilde{x}_0\|^2}{4t}-\frac{(b-m)^2}{4t}}}{\sqrt{2 \pi t}^{d}}\\
~~~~~~~~~~~~~~~~~~~~~~~~~~~& \exp [ - \frac{(b-x_0^1)^2}{4s}]\exp[ -\frac{s}{t(t-s)}\frac{(b-m)^2}{4}] .
\end{align*}
We observe that $\frac{1}{t} +\frac{s}{t(t-s)}=\frac{1}{t-s}$ so that $\exp[-\frac{(b-m)^2}{4t}]
\exp[ -\frac{s}{t(t-s)}\frac{(b-m)^2}{4}]=
\exp[ -\frac{(b-m)^2}{4(t-s)}].$
\\
We integrate  with respect to $b$ (neglecting the indicator function) using 
 Lemma \ref{lem-convol-gaussien} (ii) for $u=m,$ $v=b$ and $w=x_0^1$ { and} $\exp[-\frac{(m-x^1_0)^2}{4t}]\leq 1$:
\begin{align*}
\int_{\mathbb{R}^{d+1}}q^{k,\alpha}(m,\tilde{x},b,a;s,x_0)dadb &\leq {\mathbf 1}_{ m>x_0^1}\frac{\|B\|_{\infty} C_T D2^{(d+1)/2}}{\sqrt{t-s} }
\frac{ e^{-\frac{|\tilde{x}-\tilde{x}_0\|^2}{4t}}}{\sqrt{ t}^{d+1}}.
\end{align*}
Since $\mu_0$ is a probability measure then 
$\int_0^t\int_{\mathbb{R}^{2d+1}}q^{k,\alpha}(m,\tilde{x},b,a;s,x_0)dadb\mu_0(dx_0)ds <+\infty.$

This is \eqref{maj-lebesgue-p-aplpha} and achieves the proof of Lemma \ref{lem-cont-alpha-bis}
\end{proof}

\begin{lem}
\label{lem-cont-beta-bis}
For $k=m,1,...,d$ recall that
$$p^{k,\beta}(m,x;t)=\int_0^t\int_{\mathbb{R}^{d+1}} {\mathbf 1}_{b<m } B^{k}(a) \partial_{k}p_{W^{*1}, W} (b-a^1,x-a, t-s) p_V(m,a;s) db da ds.
$$
The map $u \mapsto p^{k,\beta}(m,m-u,\tilde{x};t)$ converges to $0$  when $u$ goes to $0^+.$
\end{lem}
\begin{proof}
Using estimation \eqref{maj-derk-p0} of $\partial_kp_{W^{*1},W}$  and  
estimation (ii) of Proposition \ref{pro-maj-pv-g} concerning  $p_V,$ 
 we dominate 
the integrand which defines $p^{k,\beta}(m,m-u,\tilde{x};t)$ by :
 $q^{k,\beta}(m,u,\tilde{x},a,b,x_0,s):=$
\begin{align*}
{\mathbf 1}_{m-u<b<m,a^1<b }
\frac{e^{-\frac{(b-a^1)^2}{4(t-s)} -\frac{(b-m+u)^2}{4(t-s)} -\frac{\|\tilde{x}-\tilde{a}\|^2}{4(t-s)} -\frac{(m-x_0^1)^2}{4s} - \frac{(m-a^1)}{4s} -\frac{\|\tilde{x}_0 -\tilde{a}\|^2}{4s}}}{\sqrt{t-s}\sqrt{2 \pi (t-s)}^{d+1}\sqrt{ 2 \pi s}^{d+1}}
\end{align*}
up to a multiplicative constant.
Meaning that
\begin{align}
\label{maj-p-beta}
\left|p^{k,\beta}(m,m-u,\tilde{x};t)\right| \leq \|B\|_{\infty}
\int_0^t \int_{{\mathbb R}^{2d+1}}
q^{k,\beta}(m,u,\tilde{x},a,b,x_0,s) db da ds \mu_0(dx_0).
\end{align}
We integrate with respect to $\tilde{a} $ using Lemma \ref{lem-convol-gaussien} (ii) with $u=\tilde{x},$ $v=\tilde{a}$ and $w=\tilde{x}_0$
\begin{align*}
\int_{R^{d-1}} q^{k,\beta}(m,u,\tilde{x},a,b,x_0,s) d\tilde{a} \leq \sqrt{2}^{d-1} \frac{e^{-\frac{\|\tilde{x}-\tilde{x}_0\|^2}{4t}}}{\sqrt{2 \pi t}^{d-1}}
\frac{e^{-\frac{(b-a^1)^2}{4(t-s)} -\frac{(b-m+u)^2}{4(t-s)}  -\frac{(m-x_0^1)^2}{4s} - \frac{(m-a^1)}{4s} }}{\sqrt{t-s}\sqrt{2 \pi (t-s)}^{2}\sqrt{ 2 \pi s}^{2}}
\end{align*}
We integrate with respect to $a^1$ between $-\infty$ and $b$ using Lemma \ref{lem-convol-gaussien} (i') 
 for $u=b,$ $v=a^{1}$ and $w=m:$
\begin{align*}
&\int_{R^{d}} {\mathbf 1}_{b>a^1}q^{k,\beta}(m,u,\tilde{x},a,b,x_0,s) da \leq
\\
& 
\frac{e^{-\frac{\|\tilde{x}-\tilde{x}_0\|^2}{4t}}}{\sqrt{2 \pi t}^{d}}
\frac{e^{-\frac{(b-m)^2}{4t} -\frac{(b-m+u)^2}{4(t-s)}  -\frac{(m-x_0^1)^2}{4s} }}{\sqrt{t-s}\sqrt{2 \pi (t-s)}\sqrt{ 2 \pi s}}\Phi_G\left(\sqrt{\frac{t}{s(t-s)2}}( b-\frac{s}{t}b-\frac{t-s}{t}m)\right)
\\
&=  
 \frac{e^{-\frac{\|\tilde{x}-\tilde{x}_0\|^2}{4t}}}{\sqrt{2 \pi t}^{d}}
\frac{e^{-\frac{(b-m)^2}{4t} -\frac{(b-m+u)^2}{4(t-s)}  -\frac{(m-x_0^1)^2}{4s} }}{\sqrt{t-s}\sqrt{2 \pi (t-s)}\sqrt{ 2 \pi s}}\Phi_G\left(\sqrt{\frac{t-s}{2st}}( b-m)\right).
\end{align*}
Since $b-m<0$, using  Lemma \ref{lem-convol-gaussien} (iii)
\begin{align*}
\int_{R^{d}}{\mathbf 1}_{b>a^1} q^{k,\beta}(m,u,\tilde{x},a,b,x_0,s) da \leq  
\frac{e^{-\frac{\|\tilde{x}-\tilde{x}_0\|^2}{4t}}}{\sqrt{2 \pi t}^{d}}
\frac{e^{-\frac{(b-m)^2}{4t} -\frac{(b-m+u)^2}{4(t-s)}  -\frac{(m-x_0^1)^2}{4s} }}{\sqrt{t-s}\sqrt{2 \pi (t-s)}\sqrt{ 2 \pi s}}e^{-\frac{t-s}{4st}( b-m)^2}.
\end{align*}
Note that $e^{-\frac{(b-m)^2}{4t}}e^{-\frac{t-s}{4st}( b-m)^2)}
=e^{-\frac{( b-m)^2)}{4s}}.$
%
We integrate this last  bound with respect to $b$ between $m-u$ and $m$ using  Lemma \ref{lem-convol-gaussien} {(i') for the triplet $(m-u,b,m)$ and the fact that  $\frac{s}{t}(m-u) +\frac{t-s}{t}m=\frac{s(m-u)+m(t-s)}{t}$}
\begin{eqnarray*}
&\int_{R^{d+1}}{\mathbf 1}_{m-u<b<m,b<a^1} q^{k,\beta}(m,u,\tilde{x},a,b,x_0,s) da db
\leq  
\frac{e^{-\frac{\|\tilde{x}-\tilde{x}_0\|^2}{4t}}}{\sqrt{2 \pi t}^{d+1}}
\frac{e^{  -\frac{(m-x_0^1)^2}{4s}}}{\sqrt{t-s}}.
\\
&\left[\Phi_G\left(\sqrt{\frac{t}{2s(t-s)}}( m-\frac{s(m-u)+m(t-s)}{t}) \right)
-  \Phi_G\left( \sqrt{\frac{t}{2s(t-s)}}( m-u-\frac{s(m-u)+m(t-s)}{t}) \right)\right].
\end{eqnarray*}
Then,
\begin{align*}
&\int_{R^{d+1}}
q^{k,\beta}(m,u,\tilde{x},a,b,x_0,s) da db
\leq 
\\
&
\frac{e^{-\frac{\|\tilde{x}-\tilde{x}_0\|^2}{4t}}}{\sqrt{2 \pi t}^{d+1}}
\frac{e^{  -\frac{(m-x_0^1)^2}{4s}}}{\sqrt{t-s}}
\left[\Phi_G\left( \sqrt{\frac{s}{2t(t-s)}}u \right)
-  \Phi_G\left(- \sqrt{\frac{t-s}{2t(t-s)}} u \right)\right].
\end{align*}
Note that
$\lim_{u\rightarrow 0} \Phi_G\left( \sqrt{\frac{s}{2t(t-s)}}u \right)
-  \Phi_G\left(- \sqrt{\frac{t-s}{2t(t-s)}} u \right)=0$ and
\\
$\left| \Phi_G\left( \sqrt{\frac{s}{2t(t-s)}}u \right)
-  \Phi_G\left(- \sqrt{\frac{t-s}{2t(t-s)}} u \right)\right| \leq 1.$
Since  $\mu$ is a probability measure, using Lebesgue dominated theorem
\begin{align*}
\lim_{u\rightarrow 0^+} \int_0^t \int_{R^{2d+1}}{\mathbf 1}_{m-u<b<m,b<a^1} q^{k,\beta}(m,u,\tilde{x},a,b,x_0,s) da db \mu_0(dx_0) ds=0.
\end{align*}
Finally  estimation \eqref{maj-p-beta} yields
$
\lim_{u\rightarrow 0^+} p^{k,\beta}(m,m-u, \tilde x;t)=0.
$
\end{proof}


\section{Cas $d=1$}
\label{sec5}
\begin{pro}
  {Let the real diffusion $X$ given by}
$dX_t=B(X_t)dt+A(X_t)dW_t$ where $A,B$ fulfil \eqref{h1h2} and \eqref{hyp:UnEl}.
Then the density of probability $p_V$  satisfies Hypothesis \ref{hyp-dens-outil}, so for any initial law
$\mu_0$ and $F\in C^2_b(\mathbb{R}^2,\mathbb{R}),$
\begin{align}
\label{PDEd1}
{\mathbb E}\left[F(M_t,X_t)\right)& ={\mathbb E} \left[ F(X_0,X_0)\right] + \int_0^t {\mathbb E} \left[ {\mathcal L}\left(F\right)(M_s,X_s)\right] ds\nonumber
\\
&+\frac{1}{2} {\int_0^t
{\mathbb E}\left[
\partial_mF(X_s,{X_s})
\|A(X_s)\|^2
\frac{p_V(X_s,{X_s};s)}
{p_{X}(X_s;s)}\right] ds.}
\end{align}
\end{pro}
\begin{proof}
We operate a Lamperti transformation \cite{lamp}. Whithout loss of generality, $A$ can be choosen positive. 
In case $d=1$ Assumption \eqref{hyp:UnEl}: ``$\exists c>0$ such that for any $x\in\mathbb{R},$
$A^2(x)\geq c$''   could be expressed:
\begin{equation}
\label{hypD1}
\exists c>0\mbox{ such that for any }x\in\mathbb{R}, A(x)\geq c.
\end{equation}
{  Let $\varphi$ such that $\varphi'=\frac{1}{A}$
and $\varphi(0)=0$, so that  $\varphi'$ is { uniformly  bounded} and  $\varphi\in C^2(\mathbb{R})$, 
 as is {the function} $A$.
Moreover $\varphi'$ being  strictly positive, $\varphi$ is  strictly increasing hence   invertible and we denote by  $\varphi^{-1}
$   its inverse function.}   {Under the initial condition  $\varphi(0)=0$, } using It\^o formula $Y=\varphi(X)$  satisfies
\begin{equation}
\label{SDEY}
dY_t=\left[\frac{B}{A}\circ\varphi^{-1}(Y_t)-\demi A'\circ\varphi^{-1}(Y_t)\right]dt+dW_t, ~{ Y_0=\varphi(X_0)}.
\end{equation}

{ Let $A_\varphi=1$ and
$B_\varphi:=\frac{B}{A}\circ\varphi^{-1}-\demi A'\circ\varphi^{-1}$ which belongs to} $ C^1_b(\mathbb{R})$
as a consequence of $B\in C^1_b,~A\in C^2_b.$
Obviously, $\varphi'>0$  implies that 
 $\varphi$ is increasing,
 $Y^*_t=\varphi(X_t^*)=\varphi(M_t)$.
 
Theorem 1.1 in
 \cite{LaureMo} can be easily extended to the case where $X$ admits an initial law
 $\mu_0$, thus the law of the pair $(Y^*_t,Y_t)$ admits a 
 density with respect to the Lebesgue measure.
 Moreover, Lemma 2.2 in \cite{LaureMo} sets out
$p_V(b,a;t)=\frac{p_{Y^*,Y}(\varphi(b),\varphi(a);t)}{A(b)A(a)}.$
Now applying Theorem \ref{existence-dens} to the pair $(B_\varphi,1)$ the density $p_{Y^*,Y}$
satisfies Hypothesis \ref{hyp-dens-outil}.
 \\
 Since functions $A$ and $\varphi$ are continuous
 $$\lim_{u\to 0^+} p_V(b,b-u;t)=\frac{p_{Y^*,Y}(\varphi(b),\varphi(b);t)}{A^2(b)}
 $$
 that means $p_V$ satisfies Hypothesis \ref{hyp-dens-outil} (ii).
 \\
 Using now  \eqref{hypD1}
 $$\sup_{u>0} p_V(b,b-u;t)\leq 
 \frac{1}{c^2}\sup_{u>0}p_{Y^*,Y}(\varphi(b),\varphi(b-u);t)$$
 and since $\varphi$ is increasing,
 {if $u>0$, $\varphi(b-u)<\varphi(b)$ and
  denoting $v=\varphi(b)-\varphi(b-u)$ it gets $v>0,$ and }
  $$\sup_{u>0} p_V(b,b-u;t)\leq 
 \frac{1}{c^2}\sup_{v>0}p_{Y^*,Y}(\varphi(b),\varphi(b)-v;t).$$
 After the change of variable $m=\varphi(b)$
 so $db=A(b)dm,$
 $$\int_0^T\int_\mathbb{R}
 \sup_{u>0}p_{Y^*,Y}(\varphi(b),\varphi(b)-u;t)dbdt=
 \int_0^T\int_\mathbb{R} A(\varphi^{-1}(m))
 \sup_{u>0}p_{Y^*,Y}(m,m-u;t){dm}dt.
 $$
 Since $A$ is bounded and $p_{Y^*,Y}$
 satisfies { Hypothesis \ref{hyp-dens-outil} (i)}, 
 \\
 $ \int_0^T\int_\mathbb{R} A(\varphi^{-1}(m))
 \sup_{u>0}p_{Y^*,Y}(m,m-u;t){dm}dt<\infty$
 and $p_V$  satisfies Hypothesis \ref{hyp-dens-outil} (i) and (ii).
 \end{proof}

 \section{Conclusion}
This paper establishes a PDE {of which} the density of the pair $[M_t, X_t]$ running maximum-diffusion process is a weak solution, under  a quite natural assumption on the regularity of $p_V$ around the boundary of $\Delta.$ This assumption is fulfilled when
 the matrix coefficient of diffusion $A$ is the identity matrix or when the dimension $d=1$.
This PDE    is degenerated then the classical
results on uniqueness cannot be applied here.
The case of non
 constant matrix $A$ is an open problem.
  Such generalization could be useful in case of practical
 applications, as the management of barrier options, in models including stochastic volatility.


\begin{appendix}
\section{Tools}
\label{Appendix}

\subsection{Malliavin calculus tools}

The material of this subsection  is taken {from} section 1.2 {in} \cite{nualart}.
\\
Let ${\mathbb H}=L^2([0,T],{\mathbb R}^d)$ endowed with the usual scalar product $\langle.,.\rangle_{{\mathbb H}}$ and the associated norm $\|.\|_{\mathbb H}.$
For all {$h \in {\mathbb H} ,$}
$W(h):= \int_0^T h(t) dW_t$
is a center Gaussian variable with variance equal to $\|h\|^2_{{\mathbb H}}.$
If {$(h_1,h_2 )\in {\mathbb H}^2,$ and} $\langle h_1,h_2\rangle_{{\mathbb H}}=0$, then, the random variables $W(h_i),~i=1,2,$  are independent.
\\
Let ${\mathcal S}$ denote the class of smooth random variables $F$ defined by:
\begin{align}
\label{smooth-va}
F=f(W(h_1),...,W(h_n)),~n\in {\mathbb N},~~
 h_1,...,h_n \in {\mathbb H},
 ~f \in C_b({\mathbb R}^n).
\end{align}

\begin{definition} The derivative of {the} smooth variable $F$ defined in (\ref{smooth-va}) is the ${\mathbb H}$ valued random variable given by
$DF:=\sum_{i=1}^n \partial_i f(W(h_1),...,W(h_n)) h_i.$
\end{definition}
We denote the domain of the operator
$D$ in $L^2(\Omega)$ by $\DD^{1,2}$ meaning that
 ${\mathbb D}^{1,2}$ is the closure of the class of smooth random variables ${\mathcal S}$ with respect to the norm
$$ \|F\|_{1,2}=\left\{ {\mathbb E}[ |F|^2] + {\mathbb E}[ \|DF\|^2_{{\mathbb H}}]\right\}^{1/2}.$$
{\begin{definition}
  ${\mathbb L}^{1,2}$ is the set of  processes $(u_s,s\in [0,T])$ which satisfy \\
  $u\in L^2(\Omega\times [0,T],{\mathbb R}^d)$ and
for all $s\in [0,T],$ $u_s$ belongs to ${\mathbb D}^{1,2}$ and
\\
$
\|u\|^2_{{\mathbb L}^{1,2}}=\|u\|_{L^{2}([0,T]\times \Omega)}^2 + \|Du\|^2_{L^2( [0,T]^2\times \Omega)}<\infty.
$
\end{definition}}
\begin{definition}
Let $u \in {\mathbb L}^{1,2},$ then { the divergence} $\delta(u)$ is the unique random variable of $L^2(\Omega)$ such that
$
{\mathbb E} \left[ F\delta(u)\right]=  {\mathbb E}\left[\langle DF,u\rangle_{{\mathbb H}} \right],~~\forall { F\in {\mathcal S}}$ smooth random variable.
\end{definition}
We { apply Definition 1.3.1 
in \cite{nualart} with  $u\in  {\mathbb L}^{1,2}$ and $G\in{\mathbb D}^{1,2}$:}
\begin{align}
\label{prop-nualart(ii}
 ~~{\mathbb E}\left[ G\delta(u)\right]={\mathbb E}\left[\langle DG,u\rangle_{{\mathbb H}}\right].
\end{align}

\noindent
{ Let $x_0\in {\mathbb R}^d.$}  We introduce  the exponential martingale
\begin{align}
\label{defZ}
{ Z_t^{x_0}}:=\exp \left[ \sum_{k=1}^d
\left(\int_0^t B^k({ x_0+}W_s) dW^k_s - \frac{1}{2}\int_0^t (B^k({ x_0+}W_s))^2ds\right)\right].
\end{align}
 When {there is no ambiguity}, we will omit the exponent $x_0.$
\begin{lem}
\label{lem-z-l12}
Let  $B\in { C^1_b({\mathbb R}^{d},{\mathbb R})},$ then { for all $x_0\in {\mathbb R}^d$}  the process 
\\
$(B(x_0+W_s)Z_s^{x_0},~~s\in [0,T])$ belongs to
${\mathbb L}^{1,2}.$
\end{lem}
\begin{proof}
 Let $x_0$ be fixed. In this proof we omit the exponent $x_0.$
Note that  $Z_t^2=$
\begin{align*}
&\exp\left(2 \sum_{k=1}^d\int_0^t B^k(x_0+W_s)dW^k_s- \frac{4}{2}\int_0^t\|B(x_0+W_s)\|^2 ds+ \frac{4-2}{2}\int_0^t\|B(x_0+W_s)\|^2 ds\right)
\\
&\leq e^{T\|B\|_{\infty}^2}\exp\left( 2 \sum_{k=1}^d \int_0^tB^k(x_0+W_s)dW^k_s- \frac{4}{2}\int_0^t\|B(x_0+W_s)\|^2 ds\right).
\end{align*}
Then, $Z_t$ belongs to $L^2(\Omega)$ for all $t\in [0,T]$ since
\begin{align}
\label{ma-moement-z2}
\sup_{t \in [0,T]} {\mathbb E} (Z_t^2) \leq e^{T\|B\|_{\infty}^2}.
\end{align}
Note that $Z_t= 1+\sum_{k=1}^d\int_0^t B^k(x_0+W_s) Z_s dW^k_s,~~t\in [0,T].$
Using Lemma 2.2.1, Theorem 2.2.1 
 of \cite{nualart}, and the   definition
  of ${\mathbb L}^{1,2},$
applied to the $\mathbb{R}^{d+1}${-valued} process $Y=(W,Z)$   with  a null drift coefficient,
the matrix ${\Sigma,~(d+1,d),}$ defined by:
\\ 
 $[\sigma^{j,k}(y), 1\leq j,k\leq d]=I_d,$
 $\sigma^{d+1,k}(y)=B^k(x_0^1+y^1,...,x_0^d+y^d)^k z,~~k=1,...,d$, 
 \\
 {we obtain that}
 $Z$  belongs to ${\mathbb L}^{1,2}.$ Since $B$ is continuously differentiable with bounded derivatives, {the process} $\left(B(W_s+x_0) Z_s,~~s\in [0,T]\right)$ belongs to
${\mathbb L}^{1,2}.$
\end{proof}
The following remark will be often used:
using line -15 page 135 of \cite{LaureMo}  or  Exercise 1.2.11 p. 36 in \cite{nualart}
\begin{align}
\label{der-sup}
D_s^1W^{1*}_t={\mathbf 1}_{[0,\tau]}(s)
\mbox{~ where~ } 
 \tau:=\inf\{s, W_s^{1*} =W^{1*}_t\}.
\end{align}

\subsection{Brownian motion case estimations}
\label{brownian}
Let us  recall the density of distribution of the pair 
 $(W^{*,1}_t,W_t^1)$,  
where $W^1$ is   a one-dimensional  Brownian motion  and $W^{*,1}$ its running maximum (see e.g., Section 3.2 in  \cite{CJY}  or {\cite{HKR}}):
\begin{align*}
p_{W^{1*},W^1}(b,a;t)=2\frac{2b-a}{\sqrt{2\pi t^3}}\exp-
\frac{(2b-a)^2}{2t}\1_{b>\sup(a,0)}.
\end{align*}

 Thus, using the independence of the components of the process $W,$ the law of $(W^{1*}_t,W_t)$ has a density with respect to the Lebesgue measure on ${\mathbb R}^{d+1}$ denoted by $p_{W^{1*},W}(.;t):$
\begin{align}
\label{dens-cas-mb}
p_{W^{1*},W}(b,a;t)= 2\frac{(2b-a^1)}{\sqrt{(2\pi)^d t^{d+2}}}e^{-\frac{(2b-a^1)^2}{2t}-\frac{\sum_{k=2}^d |a^k|^2}{2t}}{\mathbf 1}_{b\geq 0,b\geq a^1},~b\in {\mathbb R},~a=(a^1,...,a^d)\in {\mathbb R}^d.
\end{align}
\begin{lem}
\label{lem-appendice-1}
(i) For all $t{ >0},$ $p_{W^{*1},W}(.;t)$ { is the restriction to $\bar\Delta$ of a $C^{\infty}(\mathbb{R}^{d+1})$ function}
and there exists a universal constant $D$ such that { for  $x=b,a^1,a^2,...a^d,$}
\begin{align}
\label{appendix-1}
\left|\partial_xp_{W^{*1},W}(b,a;t)\right|\leq \frac{D}{\sqrt{(4\pi)^d t^{d+2}}}e^{-\frac{b^2+(b-a^1)^2}{4t} -\sum_{k=2}^d \frac{(a^k)^2}{4t}}{\mathbf 1}_{b>\max(a^1,0)}.
\end{align}
As a consequence
$
\sum_{x=b,a^1,...,a^d} \left|\partial_{x}p_{W^{*1},W}(b,a;t-s)\right|
\in L^1([0,t] \times {\mathbb R}^{d+1},dbdads).$
\begin{align*}
(ii)~~p_0(m,x;t,x_0)& \leq& \frac{e^{-\frac{(m-x^1)^2}{4t}-\frac{\|\tilde{x}-\tilde{x}_0\|^2}{4t} {- \frac{(m-x_0^1)^2}{4t} }}}{\sqrt{(2 \pi )^{d+1}t^{d+1}}}{\mathbf 1}_{m>\max(x^1,x^1_0)}\\
&= &2^{(d+1)/2}\phi_{d+1}(m-x^1,m-x_0^1,\tilde x-\tilde x_0;2t){\mathbf 1}_{m>\max(x^1,x^1_0)},
\end{align*}
\end{lem}
\begin{proof}
(i) Let $p_W$ be the density
  of  a $d$ dimensional Brownian motion, and the density of law of $W_t$ $\forall t>0:$ $p_W(.;t)\in C^\infty(\mathbb{R}^d)$:
\begin{align*}
p_W(x;t)= \frac{1}{\sqrt{2^d \pi^d t^d}}e^{- \sum_{k=1}^d \frac{(x^k)^2}{2t}},~~t>0,~~x=(x^1,...,x^d) \in {\mathbb R}^d.
\end{align*}
Its derivative with respect to $x^1$ is 
\begin{align*}
\partial_{x^1}p_W(x;t)= -\frac{x^1}{\sqrt{2^d \pi^d t^{d+2}}}e^{- \sum_{k=1}^d \frac{(x^k)^2}{2t}},~~t>0,~~x=(x^1,...,x^d) \in {\mathbb R}^d.
\end{align*}
Its second derivatives are
\begin{align*}
\partial_{x^1 x^k}^2p_W(x;t)&= \frac{x^1x^k}{\sqrt{2^d \pi^d t^{d+4}}}e^{- \sum_{k=1}^d \frac{(x^k)^2}{2t}},~~t>0,~~x=(x^1,...,x^d) \in {\mathbb R}^d, ~~k=2,...,d.
\\
\partial_{x^1 x^1}^2p_W(x;t)&= \frac{(x^1)^2-t}{\sqrt{2^d \pi^d t^{d+4}}}e^{- \sum_{k=1}^d \frac{(x^k)^2}{2t}},~~t>0,~~x=(x^1,...,x^d) \in {\mathbb R}^d.
\end{align*}
 Using  (2.1) page 106 of \cite{GarMenal} we obtain the analogous of (2.2) page 107 of \cite{GarMenal}:
there exists a constant $C$ such that 
\begin{align}
\label{maj-garoni-gaus}
|\partial_{x^1x^1 }^2p_W(x;t)| +|\partial_{x^1 x_k}^2p_W(x;t)| \leq \frac{C}{t}
p_W(x;2t),~~k=1,...,d,~~t >0,~~x\in {\mathbb R}^d.
\end{align}
{ Recall \eqref{dens-cas-mb}}
\begin{align*}
p_{W^{*1},W}(b,a;t)=2\frac{2b-a^1}{\sqrt{(2\pi)^d t^{d+2}}}e^{-\frac{(2b-a^1)^2}{2t} -\sum_{k=2}^d \frac{(a^k)^2}{2t}}{\mathbf 1}_{b\geq a^1_+},~~\forall (b,a)\in {\mathbb R}^{d+1},~~t >0.
\end{align*}
We observe
\begin{align}
  \label{pw-densgder}
p_{W^{*1},W}(b,a;t)=-2 \partial_{x^1}p_W(2b-a^1,a^2,...,a^d;t){\mathbf 1}_{b\geq a^1_+}.
\end{align}
Then $p_{W^{*1},W}(.,.;t)$ is  the restriction to $\overline{\Delta}$  of a $C^{\infty} $ function.
\\
Moreover, using the chain rule, $x$ being $(b,a^1,\cdots,a^d):$
\begin{align} 
 \label{maj-der-dens-1}
|\partial_xp_{W^{*1},W}(b,a;t)|\leq \frac{4C}{t} p_W(2b-a^1,a^2,...,a^d;2t){\mathbf 1}_{b\geq a^1_+}.
\end{align}
On the set $\{(b,a),~~b{ >}\max(0,a^1)\}$ we have
\begin{align}
\label{min-2b-a}
(2b-a^1)^2=(b+b-a^1)^2 \geq (b)^2+ (b-a^1)^2.
\end{align}
Plugging estimation \eqref{min-2b-a} into \eqref{maj-der-dens-1} yields \eqref{appendix-1} with $D=2^3 C.$
\\
(ii) Recalling the definition
$$p_0(m,x;t,x_0)=
 p_{W^{1*},W}(m-x_0^1,x-x_0;t)= 
 2\frac{m-x^1+m-x_0^1}{\sqrt{(2\pi)^d t^{d+2}}}
 e^{-\frac{(m-x^1+m-x_0^1)^2}{2t}-
 \frac{|\tilde x- \tilde x_0\|^2}{2t}}
 {\mathbf 1}_{m\geq x^1\vee x_0^1}$$
we deduce the standard bound which uses
 $xe^{-x^2}\leq e^{-x^2/2}$ and $(m-x^1+m-x_0^1)^2\geq (m-x^1)^2+(m-x_0^1)^2$:
\begin{align*}
p_0(m,x;t,x_0) &\leq \frac{e^{-\frac{(m-x^1)^2}{4t}-\frac{\|\tilde{x}-\tilde{x}_0\|^2}{4t} {- \frac{(m-x_0^1)^2}{4t} }}}{\sqrt{(2 \pi )^{d+1}t^{d+1}}}
{\mathbf 1}_{m>x^1\vee x_0^1}
\\
&=2^{(d+1)/2} \phi_{d+1}(m-x^1,m-x_0^1,\tilde x-\tilde x_0;2t){\mathbf 1}_{m>x^1\vee x_0^1},
\end{align*}
\end{proof}


%
%
%


 \begin{lem}
  \label{lem-convol-gaussien}
For all $0<s<t,$ $k\geq 1$ and all $u,v,w\in \mathbb{R}^k$
 \begin{align*}
 &(i)~~ \frac{\|u-v\|^2}{t-s}+\frac{\|v-w\|^2}{s}= \frac{t}{s(t-s)} \left\|v-\left(\frac{s}{t}u+\frac{t-s}{t} w \right)\right\|^2 + \frac{\|u-w\|^2}{t}
;
\\
&(i')~~k=1,~~\int_{-\infty}^b 
\frac{e^{-\frac{(u-v)^2}{4(t-s)}}}{\sqrt{2\pi(t-s)}}
\frac{e^{-\frac{(w-v)^2}{4s}}}{\sqrt{2\pi s}}dv=
\sqrt 2\frac{ e^{-\frac{(u-w)^2}{4t}}}{\sqrt{2\pi t)}}\Phi_G\left(\sqrt{\frac{t}{s(t-s)}}[b-(\frac{s}{t}u+\frac{t-s}{t} w)]\right)
\\
&(ii)~~\int_{\mathbb{R}^k} \frac{e^{-\frac{\|u-v\|^2}{4(t-s)}}}{\sqrt{(2\pi(t-s))^k}}\frac{e^{-\frac{\|w-v\|^2}{4s}}}{\sqrt{(2\pi s)^k}}dv= 2^{k/2}\frac{ e^{-\frac{\|u-w\|^2}{4t}}}{\sqrt{(2\pi t)^k}}
\\
&(iii) ~~\mbox{~~For~~}  u>0,~1-\Phi_G(u):= \int^{+\infty}_u \frac{e^{-\frac{z^2}{2~}}}{\sqrt{2\pi}}dz =\Phi_G(-u)\leq \frac{e^{- \frac{u^2}{2}}}{2}.
\end{align*}
\end{lem}
\proof
Point (i) is proved by a development of both hands 
then an identification of the coefficients of the squared norms and scalar products:
$\|u\|^2,\|v\|^2,\|w\|^2,u.v,u.w,v.w.$
\\
So we deduce (i') as the integral of
$$\frac{e^{-\frac{t}{4s(t-s)} \left(v-(\frac{s}{t}u+\frac{t-s}{t} w )\right)^2 - \frac{(u-w)^2}{4t}}}{\sqrt{2\pi(t-s)}\sqrt{2\pi s}}
$$
with respect to $v$ up to $b$.

(ii) is a consequence of point (i) then an integration  on $\mathbb{R}^k$ of the  { Gaussian density with respect to  $dv$.}

  (iii) { The function $u \mapsto \Phi_G(u) -\frac{e^{-\frac{u^2}{2}}}{2}$ is null at $0,$ has a null limit when $u$ goes to $ -\infty$ and its derivative is $u\mapsto -\frac{e^{-\frac{u^2}{2}}}{\sqrt{2\pi}} +u\frac{e^{-\frac{u^2}{2}}}{2}.$  Its derivative vanishes at $ \sqrt{2/\pi}$ and is negative for $u \leq \sqrt{2/\pi}$ and positive after. Then, $u \mapsto \Phi_G(u) -\frac{e^{-\frac{u^2}{2}}}{2}$ is negative for $u \leq 0.$}

\hfill$\Box$

\subsection{Proof of Remark \ref{rem-wpde}, boundary conditions of the PDE}
\label{ap3}
Here we assume that $p_V$ is regular enough.
Let $\mu_0(dx)= f_0(x)dx.$ Using Theorem  
\ref{thme-sous-hyp-dens}, 
 \eqref{PDEweakm} { means that:}
{ for all $F\in C^2_b({\mathbb R}^{d+1},{\mathbb R}^d)$}
%
%
{\footnotesize
\begin{eqnarray}
\label{depart}
&&{\int_{\bar\Delta}} F(m,x)p_V(m,x;t)dmdx =\int_{\mathbb{R}^d} F(m,m,\tilde x)f_0(m,\tilde x)dmd\tilde x+
\\
&&\int_0^t{\int_{\bar\Delta}}{\cal L}F(m,x) p_V(m,x;s)dmdxds
+{ \frac{1}{2}}
\int_0^t\int_{\mathbb{R}^d} {\|A^1(m,\tilde x)\|^2}\partial_mF(m,m,\tilde x ){ p_V(m,m,\tilde x;s)}dmd\tilde xds \nonumber
\end{eqnarray}
}
{ recalling $\LL=
 B^i\partial_{x_i}+\demi \Sigma^{ij}\partial^2_{x_i,x_j}$ where $\Sigma=AA^t$.} 

(i) { Integrating by parts with respect to a convenient $dx_k$ in 
\\
$\int_0^t{\int_{\bar\Delta}}{\cal L}F(m,x) p_V(m,x;s)dmdxds$
and noting that the support of $p_V(.,.;s)$ is $\bar\Delta, $ the boundary terms uniquely concern the component {$x^1$}:
\begin{align*}
&\int_0^t{\int_{\bar\Delta}}{\cal L}F(m,x)p_V(m,x;s)dmdxds=
-\int_0^t{ \int_{\bar\Delta}} F(m,x)\partial_{x^k}(B^kp_V)(m,x;s)dmdxds
\\
&-\demi \int_0^t{\int_{\bar\Delta}}
\partial_{x^l} F(m,x)\partial_{x^k}[
\Sigma^{{k,l}}(m,x) p_V(m,x;s)]dmdxds
\\
&
+\int_0^t \int_{{\mathbb R}^d}\left(F(m,m,\tilde x)B^1(m,\tilde x)+\demi
\partial_{x^k}F(m,m,\tilde x)\Sigma^{1,k}(m,\tilde x)\right)
p_V(m,m,\tilde x;s) dmd\tilde x ds.
\end{align*}
We again operate  an integration by parts on the second term above on the right hand:
{\footnotesize
\begin{align*}
&-\demi \int_0^t{\int_{\bar\Delta}}
\partial_{x^l} F(m,x)\partial_{x^k}
\Sigma^{{k,l}}(m,x) p_V(m,x;s)]dmdxds=
\\
&\demi \int_0^t{\int_{\bar\Delta}}
 F(m,x)\partial^2_{x^k,x^l}[\Sigma^{k,l} p_V](m,x;s)dmdxds
 { -\demi \int_0^t\int_{{\mathbb R}^d} F(m,m,\tilde{x}) \partial_{x^k} \left[ \Sigma^{1,k} p_V\right](m,m,\tilde{x};s) dm d\tilde{x} ds}.
 \end{align*}
 }
Gathering these equalities yields
\begin{eqnarray}
\label{calL*}
&&\int_0^t{\int_{\bar\Delta}}{\cal L}F(m,x)p_V(m,x;s)dmdxds
=
\int_0^t{\int_{\bar\Delta}} F(m,x)\LL^* p_V(m,x;s)dmdxds \nonumber
\\
&& { -\demi \int_0^t\int_{{\mathbb R}^d} F(m,m,\tilde{x}) \partial_{x^k} \left[ \Sigma^{1,k} p_V\right](m,m,\tilde{x};s) dm d\tilde{x} ds}
\\
&&+\int_0^t  \int_{\mathbb{R}^d}
\left(F(m,m,\tilde x)B^1(m,\tilde x)
p_V(m,m,\tilde x;s) +\demi[\partial_{x^k}
F\Sigma^{1,k} p_V](m,m,\tilde x;s)
\right) dmd\tilde{x} ds.\nonumber
\end{eqnarray}

{(ii)
Using  $F\in C^2_b({\mathbb R}^{d+1},{\mathbb R})$ with compact   support in $\Delta$ (so $F(m,m,\tilde x)=0$)  we deduce the equality in $\Delta:$
\begin{equation}
\label{EDPouvert}
\partial_tp_V(m,x;s)=\LL^* p_V(m,x;s),~~{ \forall s >0,~~(m,x) \in \Delta}.
\end{equation}
{ We use \eqref{depart}, \eqref{calL*} and  \eqref{EDPouvert} applied to 
$F\in C^2_b({\mathbb R}^{d+1},{\mathbb R})$ with compact   support in $\bar\Delta$:}
}
\begin{align}
\label{edpw1}
0&=
 \int_{\mathbb{R}^d} F(m,m,\tilde x)f_0(m,\tilde x)dmd\tilde x { -\demi \int_0^t\int_{{\mathbb R}^d} F(m,m,\tilde{x}) \partial_{x^k} \left[ \Sigma^{1,k} p_V\right](m,m,\tilde{x};s) dm d\tilde{x} ds}\nonumber
 \\
 & +\int_0^t \int_{{\mathbb R}^d}
 \left(F(m,m,\tilde x)B^1(m,\tilde x)
p_V(m,m,\tilde x;s) +\demi
 [\partial_{x^k}F\Sigma^{1,k} p_V](m,m,\tilde x;s)\right) 
 dmd\tilde{x} ds\nonumber
 \\
 & +{ \frac{1}{2}}
\int_0^t\int_{\mathbb{R}^d} {\|A^1(m,\tilde x)\|^2}\partial_mF(m,m,\tilde x ){ p(m,m,\tilde x;s)}dmd\tilde xds.
\end{align}
 We now operate integration by parts on the
last two terms:
{\footnotesize
\begin{align}
\label{edpw2}
& \int_0^t\int_{\mathbb{R}^d} \left([\partial_{x^k}F.\Sigma^{1,k}. p_V](m,m,\tilde x;s) +
{\|A^1(m,\tilde x)\|^2}\partial_mF(m,m,\tilde x ){ p_V(m,m,\tilde x;s)}\right)dmd\tilde xds=
\nonumber
\\
&- \int_0^t\int_{\mathbb{R}^d}\left([ F.\partial_{x^k}
(\Sigma^{1,k} p_V)](m,m,\tilde x;s)
+\partial_m
(\|A^1\|^2 p_V)(m,m,\tilde x;s)\right)
dmd\tilde xds
\end{align}
}
Plugging \eqref{edpw2}  into \eqref{edpw1} yields  the boundary condition, namely a PDE of which {$p_V$} is a solution in the weak sense:
\begin{eqnarray*}
&B^1(m,\tilde x) p_V(m,m,\tilde x;s)=\demi\sum_{k\geq 1}\partial_{x^k}(\Sigma^{1,k} p_V)(m,m,\tilde x;s){\mathbf +}
\\
&\demi { \sum_{k\geq 1}} \partial_{x^k}
(\Sigma^{1,k} p_V)(m,m,\tilde x;s)+\demi {{\partial_{m}(\|A^1\|^2}{ p_V})}(m,m,\tilde x;s)
\end{eqnarray*}
simplified as
$$
B^1(m,\tilde x) p_V(m,m,\tilde x;s)= { \sum_{k\geq 1}} \partial_{x^k}(\Sigma^{1,k} p_V)(m,m,\tilde x;s)+\demi {{\partial_{m}(\|A^1\|^2} p_V)}(m,m,\tilde x;s)
$$
with the initial condition $p_V(m,m,\tilde x;0)
=f_0(m,\tilde x).$
}
\hfill$\Box$

\end{appendix}

\section*{Acknowledgements}
The authors would like to thank Monique JEANBLANC  who gave us a valuable help
in writing this paper. 
The authors would like to thank also the anonymous referees for their constructive comments that improved the
quality of this paper.

\end{document}